\documentclass[11pt,reqno]{amsart}

\usepackage{lipsum}

\usepackage{graphicx}
\usepackage{array}
\usepackage{amssymb}
\usepackage{hyperref}
\usepackage{amsthm}
\usepackage{amsfonts}
\usepackage{amsmath}
\usepackage{bm}
\usepackage{mathrsfs}
\usepackage[all]{xy}
\usepackage{color}
\usepackage{subfigure}
\usepackage{tikz, tikz-cd}
\usepackage{enumerate}
\usepackage{anysize}
\usepackage{amscd}
\usepackage[letterpaper]{geometry}
\usepackage{multirow}
\usepackage{array}
\usepackage{booktabs}

\newtheorem{theorem}{Theorem}[section]
\newtheorem{proposition}[theorem]{Proposition}

\newtheorem{lemma}[theorem]{Lemma}

\newtheorem{corollary}[theorem]{Corollary}

\theoremstyle{remark}

\theoremstyle{definition}
\newtheorem{definition}[theorem]{Definition}
\newtheorem{remark}[theorem]{Remark}

\numberwithin{equation}{section}
\numberwithin{figure}{section}
\numberwithin{table}{section}

\newcommand{\bx}{\bm{x}}

\newcommand{\dN}{\mathbb{N}}
\newcommand{\dR}{\mathbb{R}}

\renewcommand{\i}{\sqrt{-1}}
\newcommand{\dT}{\mathbb{T}}
\newcommand{\dZ}{\mathbb{Z}}

\newcommand{\w}{\omega}

\newcommand{\fq}{\mathfrak{q}}

\newcommand{\fs}{\mathfrak{s}}

\newcommand{\fM}{\mathfrak{M}}

\newcommand{\RR}{\mathbb{R}}
\newcommand{\ZZ}{\mathbb{Z}}

\newcommand{\p}{\partial}

\newcommand{\cB}{\mathcal{B}}

\newcommand{\CC}{\mathbb{C}}

\DeclareMathOperator{\ALG}{ALG}
\DeclareMathOperator{\ALH}{ALH}

\DeclareMathOperator{\coker}{coker}

\DeclareMathOperator{\dvol}{dvol}
\DeclareMathOperator{\End}{End}

\DeclareMathOperator{\Id}{Id}
\DeclareMathOperator{\I}{I}
\DeclareMathOperator{\II}{II}

\DeclareMathOperator{\Image}{Image}

\DeclareMathOperator{\pr}{pr}

\DeclareMathOperator{\Ric}{Ric}

\DeclareMathOperator{\Rea}{Re}

\DeclareMathOperator{\Rm}{Rm}
\DeclareMathOperator{\SO}{SO}
\DeclareMathOperator{\SL}{SL}

\DeclareMathOperator{\Supp}{Supp}

\DeclareMathOperator{\Vol}{Vol}

\DeclareMathOperator{\Nil}{Nil}

\allowdisplaybreaks[3]

\begin{document}

 \title{Hodge theory on ALG$^*$ manifolds}

 \author{Gao Chen} \thanks{The first named author is supported by the Project of Stable Support for Youth Team in Basic Research Field, Chinese Academy of Sciences, YSBR-001. The second named author is supported by NSF Grants DMS-1811096 and DMS-2105478. The third named author is supported by NSF Grant DMS-1906265. }
  \address{Institute of Geometry and Physics, University of Science and Technology of China, Shanghai, China, 201315}
 \email{chengao1@ustc.edu.cn}
 \author{Jeff Viaclovsky}
 \address{Department of Mathematics, University of California, Irvine, CA 92697}
 \email{jviaclov@uci.edu}
 \author{Ruobing Zhang}
\address{Department of Mathematics, Princeton University, Princeton, NJ, 08544}
\email{ruobingz@princeton.edu}

\begin{abstract}
We develop a Fredholm Theory for the Hodge Laplacian in weighted spaces on ALG$^*$ manifolds in dimension four.  We then give several applications of this theory. 
First, we show the existence of harmonic functions with prescribed asymptotics at infinity. A corollary of this is a non-existence result for ALG$^*$ manifolds with non-negative Ricci curvature having group $\Gamma = \{e\}$ at infinity.  Next, we prove a Hodge decomposition for the first de Rham cohomology group of an ALG$^*$ manifold. A corollary of this is vanishing of the first betti number for any ALG$^*$ manifold with non-negative Ricci curvature. Another application of our analysis is to determine the optimal order of ALG$^*$ gravitational instantons.
\end{abstract}

\date{}

\maketitle

\setcounter{tocdepth}{1}
\tableofcontents

\section{Introduction}\label{ss:ALGstar}
In this paper, we are interested in complete Riemannian metrics in dimension $4$ which are asymptotic to certain Ricci-flat model spaces at infinity. Many of these types of geometries have been previously studied, and are known as ALE, ALF, ALG, ALH, and two exceptional types known as $\ALG^*$ and $\ALH^*$; see for example \cite{BKN, BM, CCI, Hein, HSVZ} and the references therein. In this paper, we will concentrate on the first exceptional type, $\ALG^*$, which we define next. Let $\Nil^3_{\nu}$ be the Heisenberg nilmanifold of degree $\nu \in \ZZ_+$ with coordinates $(\theta_1, \theta_2, \theta_3)$ on the universal cover; see Section~\ref{s:model-metric} for our conventions. 
\begin{definition}[ALG$^*_{\nu}$ model space] 
\label{d:ALG*-model}
For $\nu \in \ZZ_+$, the ALG$^*_{\nu}$ model manifold is 
\begin{align}
\widehat{\fM}_{\nu}(R) \equiv  (R, \infty) \times \Nil^3_{\nu}.
\end{align}
Let $V = \kappa_0 + \frac{\nu}{2\pi} \log r$, where $\kappa_0 \in \RR$, and assume that $R$ satisfies  $R >   e^{\frac{2\pi}{\nu}(1-\kappa_0)}$. The model metric on $\widehat{\fM}_{\nu}(R)$ given by
\begin{align}
\label{intro:metricexp}
g^{\widehat{\fM}}_{\kappa_0} = V ( dr^2 + r^2 d\theta_1^2 + d \theta_2^2) + V^{-1} \frac{\nu^2}{4\pi^2} \Big( d \theta_3 - \theta_2 d \theta_1 \Big)^2.
\end{align}
Given $L > 0$,  we let $g^{\widehat{\fM}}_{\kappa_0,L} = L^2 g^{\widehat{\fM}}_{\kappa_0}$.
\end{definition}
The metric in \eqref{intro:metricexp}  arises naturally from the Gibbons-Hawking ansatz; see Section~\ref{s:model-metric}.
Note that the model metric restricts to a left-invariant metric on any cross-section $\{r_0\} \times \Nil^3_{\nu}$. Next, let $\Gamma$ be any finite group acting freely and isometrically on  $\widehat{\fM}_{\nu}(R)$, and denote the quotient space by $\fM_{\nu}(R)$. We then define the following, which are the main objects of interest in this paper.  
\begin{definition}[ALG$^*_{\nu}$-$\Gamma$ manifold]
\label{d:ALGstar} 
A complete $4$-manifold $(X,g)$
is called an ALG$^*_{\nu}$-$\Gamma$ manifold of order $\mathfrak{n} > 0$ with parameters 
$\nu \in \ZZ_+$, $\kappa_0 \in \RR$ and $L \in \RR_+$ if there exist an ALG$^*$ model space $(\widehat{\fM}_{\nu}(R), g^{\widehat{\fM}}_{\kappa_0,L})$  with $R >0$, a compact subset $X_R \subset X$, and a diffeomorphism 
$\Phi: \widehat{\fM}_{\nu}(R) /\Gamma \rightarrow X \setminus X_R$ such that
  \begin{align}
\label{ALGstarmetric}
\big|\nabla^k_{g^{\fM}}(\Phi^*g- g^{{\fM}}_{\kappa_0,L})\big|_{g^{\fM}}=O(\fs^{-k-\mathfrak{n}}),
  \end{align}
as $\fs(x) \equiv r(x)V^{\frac{1}{2}}(x)\to \infty$, for any $x \in \widehat{\fM}_{\nu}(R)$, for any $k\in\dN_0 \equiv \ZZ_+ \cup \{0\}$.
\end{definition}
We note that allowing the quotient by $\Gamma$ is analogous to the asymptotically flat (AF) case versus the asymptotically locally Euclidean (ALE) case. 
Fix $x_0 \in X_R$, and for $x \in X$, define $s(x) \equiv d_g(x_0,x)$.
We also note that ALG$^*_{\nu}$-$\Gamma$ metrics satisfy the following properties:
\begin{enumerate}
\item   $s(x) \sim  \fs(\Phi^{-1}(x))$ as $s(x) \to \infty$;  see Remark~\ref{r:scomp},

\item  $\Vol_g(B_t(x_0)) \sim t^2$ as $t \to \infty$, and

\item  $|\Rm_g| = O( s^{-2} (\log s)^{-1})$ as $s(x) \to \infty$. 
\end{enumerate}

Fredholm Theory of the Hodge Laplacian for the geometries listed above has been developed in many works; see for example \cite{Bartnik, CCI, CCII, CCIII, HHM, HSVZ, Mazzeo, MinerbeALF}. It is also worth adding a historical remark regarding the ALG$^*$ model geometry. ALG$^*$ hyperk\"ahler manifolds appeared in the math literature before the special case of ALG manifolds. In fact, Cherkis-Kapustin first studied the ALG$^*$ model geometry, ALG$^*$ hyperk\"ahler manifolds, as well as D$_4$-ALG hyperk\"ahler manifolds; see \cite{CherkisKapustinALG}. The ALG name was introduced in that paper (page~2, last paragraph), which included both the ALG and the ALG$^*$ cases. Later, these cases were separated in order to distinguish the different model geometries.

In this paper, we will develop a Fredholm Theory for the Hodge Laplacian on ALG$^*_{\nu}$-$\Gamma$ manifolds; see Proposition~\ref{Weighted-analysis-X}. Due to the peculiarities of the asymptotic geometry of ALG$^*$ metrics, this is quite nontrivial; it is proved in Sections \ref{s:weighted-model} and \ref{s:weighted}, which rely on some lengthy formulas computed in the appendix. We will not describe this theory here in the introduction, but instead we will turn our attention to a number of applications. Note that $h_0 \equiv 1$ is obviously a harmonic function. Our first application is to the existence of non-constant harmonic functions with prescribed asymptotics. 
\begin{theorem}
\label{t:ehf}
 Let $(X,g)$ be an $\ALG^*_{\nu}$-$\Gamma$ manifold of order $\mathfrak{n} > 0$, 
and $k \in \dZ_+$. If  $r^k e^{\i k \theta_1}$ is invariant under $\Gamma$, then for any $0<\epsilon<\min\{k,\mathfrak{n}\}$, there exists a harmonic function $h_k : X \rightarrow \CC$ such that 
\begin{equation}
\label{e:hkexp}
h_k = r^k e^{\i k \theta_1} + O(s^{k - \epsilon}),
\end{equation} as $s \to \infty$.
\end{theorem}
A corollary of this is the following non-existence result. 
\begin{corollary} 
\label{t:non-existence}
There do not exist any $\ALG^*_{\nu}$-$\Gamma$ manifolds of order $\mathfrak{n} > 0$ with $\Gamma = \{e\}$ and with non-negative Ricci curvature.
\end{corollary}
To prove this, we will use Theorem \ref{t:ehf} to find a certain harmonic $1$-form which, using the Bochner formula, will lead to a contradiction. 
We note that there do exist complete ALG$^*_{\nu}$-$\Gamma$ metrics  with non-negative Ricci curvature for $\Gamma$ nontrivial. A rough analogy with the AF setting: any AF metric with non-negative Ricci curvature must be Euclidean space. But there are many non-flat ALE metrics with $\Gamma$ non-trivial and with non-negative Ricci curvature. 

Our next application is a Hodge decomposition theorem for the first de Rham cohomology group. Define the following subspace of $\Omega^1(\widehat{\fM}_{\nu}(R) / \Gamma)$:
\begin{align}
\mathcal{W}^1 \equiv
\begin{cases}
 \RR \cdot d \theta_2,  & \mbox{ if } \gamma^* d \theta_2 = d \theta_2  \mbox{ for all } \gamma \in \Gamma,\\
 \{0\}, & \mbox{ otherwise}.
\end{cases}
\end{align}

\begin{theorem}
\label{t:h1}
Let $(X,g)$ be an $\ALG^*_{\nu}$-$\Gamma$ manifold. Then
\begin{align}
\label{e:h1dr}
H^1_{dR}(X) \cong \{ \omega \in \Omega^1(X) \ | \ d\omega =0, \delta \omega= 0, \Phi^*\omega = \omega_0 + O(\fs^{-\epsilon}) \mbox{ as }\fs \to \infty , \ \omega_0 \in \mathcal{W}^1\},
\end{align}
for any $\epsilon$ satisfying $0 < \epsilon \ll 1$.  
\end{theorem}
The proof of this is found in Section \ref{s:hodge}. There is an analogous result for the second de Rham cohomology group which can be proved using similar methods, but for simplicity we do not state this in this paper.  As a corollary of Theorem~\ref{t:h1}, we have the following vanishing theorem. 
\begin{corollary}
\label{t:h10}
Let $(X,g)$ be an $\ALG^*_{\nu}$-$\Gamma$ manifold with non-negative Ricci curvature. Then the first betti number $b^1(X) = 0$.
\end{corollary}
We note that since any such manifold has quadratic volume growth at infinity, this corollary can also be seen to follow from Theorem 2.1 and Lemma 2.2 in  \cite{Anderson}. Anderson's results are much more general, but his proof uses techniques from geometric measure theory. In contrast, our proof of Corollary \ref{t:h10} is more elementary, but is specialized to $\ALG^*$ geometry. 

We next recall the following definition. 
\begin{definition}A hyperk\"ahler $4$-manifold $(X, g, I, J, K)$ is a Riemannian $4$-manifold $(X,g)$ with a triple of K\"ahler structures $(g, I), (g, J), (g, K)$ such that $IJ=K$. Let $\omega_1, \omega_2, \omega_3$ denote the K\"ahler forms for $I,J,K$, respectively.
\end{definition}
As we will see in Subsection~\ref{ss:HK}, the model space has a hyperk\"ahler structure, and we denote the triple of K\"ahler forms by $\omega_{i, \kappa_0, L}^{\widehat{\fM}}, i = 1,2,3$. In Subsection \ref{ss:quotients}, we consider all possible isometric quotients of the model space which retain this hyperk\"ahler structure. We show that, without loss of generality, we may assume that either $\Gamma$ is trivial, or $\nu$ is even and $\Gamma$ is a specific $\ZZ_2$-action $\iota: \widehat{\fM}_{2\nu}(R) \rightarrow  \widehat{\fM}_{2\nu}(R)$, up to hyperk\"ahler rotation and scaling.
In the latter case, denote the quotient space by 
$\fM_{2\nu}(R)=  \widehat{\fM}_{2\nu}(R)/\ZZ_2$. The model K\"ahler forms descend to the quotient, which we denote by $\omega_{i, \kappa_0, L}^{\fM}, i = 1,2,3$. 

Next, we will consider the case of complete hyperk\"ahler $\ALG^*_{\nu}$-$\Gamma$ manifolds. 
Since any hyperk\"ahler metric is Ricci-flat, by Corollary~\ref{t:non-existence}, $\Gamma$ cannot be trivial. 
Consequently, by the remarks in the previous paragraph, we can make the following definition. 

\begin{definition}
\label{d:ALGstarGI} An $\ALG^*_{\nu}$  gravitational instanton $(X,g,I,J,K)$  of order $\mathfrak{n} > 0$ with parameters 
$\nu \in \ZZ_+$, $\kappa_0 \in \RR$ and $L \in \RR_+$ is a hyperk\"ahler $4$-manifold
which is also an $\ALG^*_{2\nu}$-$\ZZ_2$ manifold $(X,g)$ with the $\ZZ_2$-action given by $\iota$, 
such that in addition to Definition~\ref{d:ALGstar},
the hyperk\"ahler forms satisfy
\begin{align}
\label{algstaromega}
 \big|\nabla_{g^{\fM}}^k(\Phi^*\omega_i- L^2 \omega_{i,\kappa_0,L}^{\fM})\big|_{g^{\fM}}=O(\fs^{-k-\mathfrak{n}}),
  \end{align}
for $i = 1,2,3$, as $\fs (x) \equiv r(x)V^{\frac{1}{2}}(x) \to \infty$, for any $k\in\dN_0$. \end{definition}
\begin{remark} The conditions \eqref{algstaromega} necessarily imply \eqref{ALGstarmetric}; see Subsection~\ref{s:aALGstar} below. 
\end{remark}

Our next main application is to determine the optimal order of $\ALG^*_{\nu}$ gravitational instantons. 
\begin{theorem}\label{t:sharp-order-intro} If $(X,g,I,J,K)$ is an $\ALG^*_{\nu}$ gravitational instanton of order $\mathfrak{n} > 0$ with respect to $\ALG^*_{\nu}$ coordinates $\Phi$, then there exists an $\ALG^*_{\nu}$ coordinate system $\Phi'$ as in Definition \ref{d:ALGstarGI} with order $\mathfrak{n}' = 2$. 
\end{theorem}
This is proved as an application of our Fredholm Theory applied to a certain Dirac-type operator and for the Hodge Laplacian on $2$-forms; see Section \ref{s:aALGstar}.

\subsection{Acknowledgements} The authors would like to thank Hans-Joachim Hein, Rafe Mazzeo, and Song Sun for numerous helpful discussions. We also would like to thank the anonymous referee for their helpful questions and insightful remarks which greatly improved the exposition of the paper.

\section{ALG$^*$ model space}
\label{s:model-metric}
In this section, we explain some properties of $\ALG^*$ metrics in more detail.

\subsection{The model metric}
\label{ss:model-metric}
\renewcommand{\t}[1]{\theta_{#1}}
The $3$-dimensional Heisenberg group is
\begin{equation}
\label{e:3dh}
H(1,\dR) \equiv\left\{\begin{bmatrix}
1 & \t2 & \t3 \\
0 & 1 &  \t1 \\
0 & 0 & 1
\end{bmatrix}: \ \t1 ,\t2, \t3 \in\dR\right\}.
\end{equation}
For $\nu \in \dZ_+$,
the Heisenberg nilmanifold $\Nil^3_{\nu}$ of degree $\nu$
is the quotient of $H(1,\dR)$ by the left action  of the subgroup
\begin{equation}
H(1,\ZZ) \equiv\left\{\begin{bmatrix}
1 & 2 \pi k & 4 \pi^2 \nu^{-1} l \\
0 & 1 & 2 \pi m \\
0 & 0 & 1
\end{bmatrix}: \ k,l,m \in \ZZ\right\}
\label{subgroup}
\end{equation}
generated by
\begin{align}
\label{group1}
\sigma_1 (\theta_1,\theta_2 , \theta_3) &\equiv (\theta_1 + 2 \pi, \theta_2, \theta_3 ),\\
\label{group2}
\sigma_2(\theta_1,\theta_2 , \theta_3)&\equiv (\theta_1, \theta_2+ 2 \pi, \theta_3 + 2 \pi \theta_1),\\
\label{group3}
\sigma_3(\theta_1,\theta_2 , \theta_3)&\equiv ( \theta_1, \theta_2, \theta_3 +  4 \pi^2 \nu^{-1} ).
\end{align}
Note that the forms
\begin{align}\label{e:wtf}
d\theta_1, \quad d\theta_2, \quad \Theta \equiv \frac{\nu}{2\pi} (d \theta_3 - \theta_2 d\theta_1)
\end{align}
are a basis of left-invariant $1$-forms.  Also, it is clear that $\Nil^3_{\nu}$ is the total space of a degree $\nu$ circle fibration
\begin{align}
\label{cirfib}
  S^1 \longrightarrow \Nil^3_{\nu} \xrightarrow{\ \pi \ } T^2 \equiv \dR^2_{\theta_1,\theta_2}/\Lambda.
\end{align}
We next consider the Gibbons-Hawking ansatz 
\begin{align}
S^1 \to \widehat{\fM}_{\nu}(R) \equiv (R, \infty) \times \Nil^3_{\nu} \to \tilde{U} \equiv (\RR^2 \setminus \overline{B_{R}(0)}) \times S^1,
\end{align}
with the radial harmonic function 
\begin{align}V = \kappa_0 + \frac{\nu}{2\pi} \log r, \quad r\in(R,\infty),\ \kappa_0\in\dR, \ R>e^{\frac{2\pi}{\nu}(1-\kappa_0)}; \end{align}
for details of the Gibbons-Hawking ansatz construction see \cite{GibbonsHawking, HKLR}.
We use the coordinates $(x, y, \theta_2) = (r \cos (\theta_1), r \sin (\theta_1) ,\theta_2)$ on $(\RR^2 \setminus \overline{B_{R}(0^2)})\times S^1$  and fix the orientation $ r dr  \wedge d \theta_1 \wedge d \theta_2$.
Then we have 
\begin{align}
\label{e:stardV}
*_{\dR^2\times S^1}\circ d(V) =  \frac{\nu}{2\pi}  d \theta_1 \wedge d \theta_2,
\end{align}
and hence $\frac{1}{2\pi}[*_{\dR^2\times S^1}\circ dV] \in H^2(\tilde{U}; \dZ)$. Note that the form $\Theta$ is a connection form such that
$\Omega = d \Theta = * dV$. The Gibbons-Hawking metric is
\begin{align}
\label{metricexp}
g^{\widehat{\fM}}_{\kappa_0} = V ( dx^2 + dy^2 + d\theta_2^2) + V^{-1} \Theta^2 = V ( dr^2 + r^2 d\theta_1^2 + d \theta_2^2) + V^{-1} \frac{\nu^2}{4\pi^2} \Big( d \theta_3 - \theta_2 d \theta_1 \Big)^2.
\end{align}
If $\Gamma$ is some specified finite group acting freely and isometrically on $\widehat{\fM}_{\nu}(R)$, then we will denote the quotient space by ${\fM}_{\nu}(R) = \widehat{\fM}_{\nu}(R)/\Gamma$.

\begin{remark}
\label{r:scomp}
Choose a point $p_0\in \widehat{\fM}_{\nu}(R)$.
By straightforward computations, one can see that there exists a constant $C>0$ such that for any $q\in \widehat{\fM_{\nu}}(R)$,
\begin{align}
C^{-1}\cdot r(q) \cdot V(q)^{\frac{1}{2}}\leq d_{g^{\widehat{\fM}}_{\kappa_0}}	(q, p_0) \leq C \cdot r(q) \cdot V(q)^{\frac{1}{2}}.
\end{align}
\end{remark}
\subsection{Hyperk\"ahler structure}
\label{ss:HK}
On $\widehat{\fM}_{\nu}(R)$, define an orthonormal basis
\begin{align}
\{ E^1, E^2, E^3, E^4 \} = \{ V^{1/2} dx, V^{1/2} dy , V^{1/2} d\t2, V^{-1/2} \Theta\}.
\end{align}
We define $3$ almost complex structures $I, J, K \in C^\infty(\End(T \widehat{\fM}))$ on $\widehat{\fM}_{\nu}(R)$ by requiring the dual linear maps $I^*, J^*, K^* \in C^\infty(\End(T^* \widehat{\fM}))$ satisfying
\begin{align}
I^* ( E^1 ) &= - E^2, \quad I^* (E^3) = - E^4, \\
J^* ( E^1) &= - E^3, \quad J^*(E^2) = E^4, \\
K^* ( E^1) &= - E^4, \quad K^*(E^2) = - E^3.
\end{align}
It is clear that each complex structure is Hermitian with respect to $g$. Moreover, $K=IJ$ because $K^*=J^*I^*$.
Using the convention that $\omega_J(X,Y) = g (JX,Y)$, the corresponding
K\"ahler forms are
\begin{align}
\label{mkf1}
\omega_I&=\omega_{1, \kappa_0}^{\widehat{\fM}}  = E^1 \wedge E^2 + E^3 \wedge E^4= V dx \wedge dy + d \t2 \wedge \Theta\\
\label{mkf2}
\omega_J&= \omega_{2, \kappa_0}^{\widehat{\fM}} = E^1 \wedge E^3 - E^2 \wedge E^4 = V dx \wedge d \t2 - dy \wedge \Theta\\
\label{mkf3}
\omega_K&=\omega_{3, \kappa_0}^{\widehat{\fM}} = E^1 \wedge E^4 + E^2 \wedge E^3 = dx \wedge \Theta + V dy \wedge d \t2.
\end{align}
We notice that $d\omega_I = d\omega_J = d\omega_K = 0$. For example, 
\begin{align}
d\omega_I = dV \wedge *d\theta_2 - d\theta_2 \wedge d\Theta = dV \wedge *d\theta_2 - d\theta_2 \wedge *dV = 0,
\end{align}
where $*\equiv *_{\dR^2\times S^1}$ is the Hodge star operator defined with respect to the flat metric on $\dR^2\times S^1$. The last equality holds since $* \alpha \wedge \beta = \alpha \wedge * \beta$ for $1$-forms $\alpha,\beta\in\Omega^1(\mathbb{R}^2\times S^1)$.
The computations for $d\omega_J$ and $d\omega_K$ are similar.

By \cite[Lemma~6.8]{Hitchin}, the triple $ \{ \omega_{1, \kappa_0}^{\widehat{\fM}}, \omega_{2, \kappa_0}^{\widehat{\fM}}, \omega_{3,\kappa_0}^{\widehat{\fM}} \} \equiv \{ \omega_I, \omega_J, \omega_K \}$
 is a hyperk\"ahler triple on $\widehat{\fM}_{\nu}(R)$.  In particular, the complex structures $I,J,K$ are integrable, and the metric is Ricci-flat K\"ahler with respect to all three of these complex structures.

\subsection{Quotients}
\label{ss:quotients}
There are many possibilities for the group $\Gamma$; see for example \cite{ChuShin}.  We will not analyze all the possibilities here, but will only address the question of which quotients retain the hyperk\"ahler structure. To this end, we have the following proposition. 
\begin{proposition} 
\label{p:groups}
If $\Gamma$ is a finite group acting freely and isometrically on 
$(\widehat{\fM}_{\nu}(R), g_{\kappa_0}^{\widehat{\fM}})$ and preserving $I, J, K$, then either $\Gamma$ is generated by $\xi_l$ and $\zeta_{k, l, m}$ for integers $k, l \in \dN ,m \in \dN_0$ such that $k$ divides $\nu$ and $0 \leq m \leq kl-1$,   or $\nu$ is even and $\Gamma$ is generated by $\xi_l$, $\zeta_{k, l, m}$, $\iota_{n, t}$ for integers $k,l,m$ satisfying the same conditions, $n \in \dN_0$ satisfying $0 \leq n \leq \nu - 1$ with $nl$ even,   and $t \in \RR$.
\end{proposition}
\begin{proof}
Define
\begin{align}
\label{e:xil}
\xi_l(\t1, \t2, \t3) &\equiv (\t1, \t2, \t3 + 4\pi^2 l^{-1}\nu^{-1}),\\
\label{e:zetak}
\zeta_{k, l, m}(\t1, \t2, \t3) &\equiv (\t1, \t2 + 2\pi k^{-1}, \t3 + 2\pi k^{-1}\t1 + 4\pi^2 m l^{-1} k^{-1} \nu^{-1}),\\
\iota_{n, t}(\t1, \t2, \t3) &\equiv (\t1 + \pi, 2 \pi n \nu^{-1} -\t2, 2 \pi n \nu^{-1} \t1 + t -\t3).
\end{align}
Then
\begin{align}
\xi_l \sigma_1 \xi_l^{-1} &=\sigma_1, \quad \xi_l \sigma_2 \xi_l^{-1} =\sigma_2, \quad \xi_l \sigma_3 \xi_l^{-1} =\sigma_3,\\
\zeta_{k, l, m} \sigma_1 \zeta_{k, l, m}^{-1} &= \sigma_1 \sigma_3^{\nu/k}, \quad \zeta_{k, l, m} \sigma_2 \zeta_{k, l, m}^{-1} = \sigma_2, \quad \zeta_{k, l, m} \sigma_3 \zeta_{k, l, m}^{-1} = \sigma_3,\\
\iota_{n, t} \sigma_1 \iota_{n, t}^{-1}&=\sigma_1\sigma_3^{n}, \quad \iota_{n, t} \sigma_2 \iota_{n, t}^{-1}=\sigma_3^{\nu/2}\sigma_2^{-1}, \quad \iota_{n, t} \sigma_3 \iota_{n, t}^{-1}=\sigma_3^{-1}.
\end{align}
These imply that $\xi_l$, $\zeta_{k,l,m}$, $\iota_{n,t}$ descend to actions on $H(1,\RR) / H(1,\ZZ)$. Moreover, it is easy to see that they induce actions on $\widehat{\fM}$ which fix $(g_{\kappa_0}^{\widehat{\fM}}, \omega_{1, \kappa_0}^{\widehat{\fM}}, \omega_{2, \kappa_0}^{\widehat{\fM}}, \omega_{3,\kappa_0}^{\widehat{\fM}})$. Note that 
\begin{align}
\xi_l^l =\sigma_3, \quad \zeta_{k, l, m}^{k} = \xi_l^m \sigma_2, \quad \iota_{n, t}^2=\xi_l^{nl/2}\sigma_1.
\end{align}
A routine calculation shows that they generate a finite subgroup.

Conversely, an isometry of the Gibbons-Hawking metric must map $S^1$-fibers to $S^1$-fibers, so there is a homomorphism $h$ from $\Gamma$ to the isometry group of $ (\RR^2 \setminus B_R(0^2)) \times S^1$. We first consider the kernel of $h$. Since $\Gamma$ is a finite group, we see that the kernel of $h$ is generated by $\xi_l$, for some $l \in \dN$. For the image of $h$,  the isometries of the base are rotations and reflections in $\RR^2$ and similarly on~$S^1$.  Using \eqref{mkf2} and \eqref{mkf3}, we see that $\omega_J$ and $\omega_K$ are invariant only if 
\begin{align}
\label{case1}
(dx, dy, d \theta_2, \Theta) &\mapsto (dx, dy, d \theta_2, \Theta), \mbox{ or}\\
\label{case2}
(dx, dy, d \theta_2, \Theta) &\mapsto (-dx, -dy, -d \theta_2, -\Theta).
\end{align}
Consider the case that every element in $\Gamma$ satisfies \eqref{case1}, the only possibility is that $h(\Gamma)$ is generated by 
\begin{align}
h(\gamma): (r, \t1, \t2) &\mapsto (r, \t1, \t2 + 2\pi a),
\end{align}
where $a \in \RR$. Since $\Gamma$ is a finite group, we can assume that $a=1/k$, where $k \in \dN$. Using the condition that $\Theta$ is fixed under the action, we see that
\begin{align}
\gamma (r, \t1, \t2, \t3) = (r, \t1, \t2 + 2\pi k^{-1}, \t3 + 2\pi k^{-1} \t1 + b),
\end{align}
where $b \in \RR$.
Since $\gamma \sigma_1 \gamma^{-1}$ is in $H(1, \ZZ)$, we see that $k$ divides $\nu$. Moreover, since $\gamma^k$ is in the kernel of $h$, it must be $\xi_l^m \sigma_2$, where $m \in \ZZ$. This implies that $b = 4\pi^2 m l^{-1} k^{-1} \nu^{-1}$. By multiplying with $\sigma_3$, we can assume that $0 \leq m \leq kl-1$.

Next, consider that case that \eqref{case2} happens for some element $\gamma \in \Gamma$. 
Then 
\begin{align}
\gamma (r, \t1, \t2, \t3) = (r, \t1 + \pi, c - \t2, c \t1 + t - \t3),
\end{align}
where $c, t \in \RR$. Using $\gamma \sigma_1 \gamma^{-1}\in H(1,\ZZ)$, we see that $c= 2\pi n \nu^{-1}$ for $n \in \ZZ$. By multiplying with $\sigma_2$, we can assume that $0 \leq n \leq \nu-1$. Using $\gamma \sigma_2 \gamma^{-1} \in H(1,\ZZ)$, we see that $\nu$ is even. Finally, since $\gamma^2 =\xi_l^{nl/2}\sigma_1$ is in the kernel of $h$, we see that $nl$ is also even.
\end{proof}
\begin{remark} 
\label{r:groups}
Note that in the first case in Proposition \ref{p:groups}, $h(\Gamma)$ is a cyclic group $\ZZ_k$ consisting of rotations of the $S^1$ factor. The kernel of $h$ is also a cyclic group $\ZZ_l$. There is a short exact sequence
\begin{equation}
\begin{tikzcd}
\label{cd1}
0 \arrow[r]  & \ZZ_l   \arrow[r] &  \Gamma \arrow[r, "h"] & \ZZ_k  \arrow[r] & 0,
\end{tikzcd}
\end{equation}
so we must have $\Gamma = \ZZ_l \rtimes \ZZ_k$, and since $\Gamma$ is abelian, we must have $\Gamma = \ZZ_l \times \ZZ_k$. In the second case, $h(\Gamma)$ is a dihedral group $D_{k} = \ZZ_{k} \rtimes \ZZ_2$, and we similarly conclude that 
$\Gamma = \ZZ_l \rtimes D_{k}$. In fact, it is not hard to see from the above presentation that $\Gamma = (\ZZ_l \times \ZZ_{k}) \rtimes \ZZ_2$.
\end{remark} 

\begin{remark} Recall that in Definition~\ref{d:ALG*-model}, we defined 
the scaled metric $g^{\widehat{\fM}}_{\kappa_0, L} = L^2 g^{\widehat{\fM}}_{\kappa_0}$.
The complex structures of course do not depend on any scaling, but 
we also define the rescaled K\"ahler forms 
$\omega^{\widehat{\fM}}_{i,\kappa_0, L} \equiv L^2 \omega^{\widehat{\fM}}_{i,\kappa_0}$ for $i = 1, 2, 3$.
\end{remark}

The next proposition deals with the first case above.  
\begin{proposition} In the first case in Proposition~\ref{p:groups}, 
the quotient space 
\begin{align}
\label{e:qs}
(\widehat{\fM}_{\nu}(R), g_{\kappa_0}^{\widehat{\fM}}, \omega_{1, \kappa_0}^{\widehat{\fM}}, \omega_{2, \kappa_0}^{\widehat{\fM}}, \omega_{3,\kappa_0}^{\widehat{\fM}})/\Gamma
\end{align}
can be identified with $(\widehat{\fM}_{\tilde{\nu}}(\tilde{R}), g_{\tilde{\kappa}_0,L}^{\widehat{\fM}}, \omega_{1, \tilde{\kappa}_0,L}^{\widehat{\fM}}, \omega_{2, \tilde{\kappa}_0,L}^{\widehat{\fM}}, \omega_{3,\tilde{\kappa}_0,L}^{\widehat{\fM}})$ after a hyperk\"ahler rotation, for some $\tilde{\nu}, \tilde{R}, \tilde{\kappa}_0$, and $L$. 
\label{t:no-cylic-quotient}
\end{proposition}
\begin{proof} Letting $\tilde{\theta}_1 = \t1 + 2\pi m l^{-1}\nu^{-1}$, $\tilde{\theta}_2 = k \t2$ and $\tilde{\theta}_3 = k \t3$, the
Gibbons-Hawking metric becomes
\begin{align}
\begin{split}
\label{metricexp2}
g^{\widehat{\fM}}_{\kappa_0} &= V ( dr^2 + r^2 d\tilde{\theta}_1^2 + k^{-2} d \tilde{\theta}_2^2) + V^{-1} \frac{\nu^2}{4\pi^2k^2}(  d \tilde{\theta}_3 - \tilde{\theta}_2  d \tilde{\theta}_1 )^2\\
&= k^{-1} l^{-1} ( V k^{-1} l) \Big( k^2(dr^2 + r^2 d \tilde{\theta}_1^2) + d \tilde{\theta}_2^2 \Big) + k^{-1}l^{-1}(V k^{-1} l)^{-1} \frac{(\nu k^{-1} l)^2}{4\pi^2}(  d \tilde{\theta}_3 - \tilde{\theta}_2 d \tilde{\theta}_1)^2.
\end{split}
\end{align}
Letting $\tilde{V} = V k^{-1} l$, $\tilde{r} = k r$, $\tilde{\nu} = \nu k^{-1} l$, we have 
\begin{align}
\begin{split}
\label{metricexp3}
g^{\widehat{\fM}}_{\kappa_0}  &= k^{-1}l^{-1}\Big\{  \tilde{V} ( d\tilde{r}^2 + \tilde{r}^2 d\tilde{\theta}_1^2 + d \tilde{\theta}_2^2) + \tilde{V}^{-1} \frac{\tilde{\nu}^2}{4\pi^2} (d \tilde{\theta}_3 - \tilde{\theta}_2 d \tilde{\theta}_1)^2 \Big\}.
\end{split}
\end{align}
A similar calculation shows that
\begin{align}
\omega_I = k^{-1}l^{-1} \Big\{
\tilde{V} \tilde{r} d \tilde{r} \wedge d\tilde{\theta}_1 + \frac{\tilde{\nu}}{2 \pi} d \tilde{\theta}_2 \wedge ( d \tilde{\theta}_3 - \tilde{\theta}_2 d \tilde{\theta}_1) \Big\}.
\end{align}
We also see that 
\begin{align}
\omega_J &= k^{-1}l^{-1} \big( \cos(q) \tilde{\omega}_J - \sin(q) \tilde{\omega}_K\big),\\
\omega_K &= k^{-1} l^{-1}  \big( \sin(q) \tilde{\omega}_J + \cos(q) \tilde{\omega}_K\big),
\end{align}
for $q = - 2\pi m l^{-1} \nu^{-1}$, where $\tilde{\omega}_J$ and $\tilde{\omega}_K$ are defined as in \eqref{mkf2}-\eqref{mkf3}, but with respect to the $(\tilde{r}, \tilde{\theta}_1, \tilde{\theta}_2, \tilde{\theta}_3)$-coordinates. This finishes the proof. 
\end{proof}
The next proposition deals with the second case above. 
\begin{proposition}
In the second case in Proposition~\ref{p:groups}, 
the quotient space \eqref{e:qs}
can be identified with 
\begin{align}
(\widehat{\fM}_{\tilde{\nu}}(\tilde{R}), g_{\tilde{\kappa}_0,L}^{\widehat{\fM}}, \omega_{1, \tilde{\kappa}_0,L}^{\widehat{\fM}}, \omega_{2, \tilde{\kappa}_0,L}^{\widehat{\fM}}, \omega_{3,\tilde{\kappa}_0,L}^{\widehat{\fM}})/ \iota
\end{align}
after a hyperk\"ahler rotation, for some $\tilde{\nu}, \tilde{R}, \tilde{\kappa}_0, L$, where $\tilde{\nu}$ is even and $\iota = \iota_{0,0}$.
\end{proposition}
\begin{proof}In this case, from Remark \ref{r:groups}, we have
$\Gamma = (\ZZ_l \times \ZZ_{k}) \rtimes \ZZ_2$.
Then $\ZZ_2$ acts on $\widehat{\fM}_{\nu}(R)/G$, where $G = \ZZ_l \times \ZZ_{k}$.  
Letting $\varphi(\t1, \t2, \t3) = ( \t1 + 2 \pi m l^{-1} \nu^{-1} , k \t2, k \t3)$, we see that
\begin{align}
\varphi^{-1} \iota_{n,t} \varphi = \iota_{\tilde{n},\tilde{t}}
\end{align}
for some $\tilde{n} \in \ZZ$ and $\tilde{t} \in \RR$. So we can assume that we are in the second case of Proposition \ref{t:no-cylic-quotient} with  $\tilde{k}=\tilde{l}= 1$ and $\tilde{m}=0$. Then we simply define 
\begin{align}
\hat{\theta}_2 \equiv \tilde{\theta}_2 - \pi \tilde{n} \tilde{\nu}^{-1}, \quad  \hat{\theta}_3 \equiv \tilde{\theta}_3 - \pi \tilde{n} \tilde{\nu}^{-1} \tilde{\theta}_1 - \frac{\tilde{t}}{2}
\end{align}
to change $\tilde{n}$ and $\tilde{t}$ to $0$.
\end{proof}
\begin{remark}The level sets $\{r = r_0\}$ on $\fM_{2\nu}(R)=\widehat{\fM}_{2\nu}(R) / \iota$ are
non-orientable line-bundles over Klein bottles, and are
\textit{infranilmanifolds}, which are double covered by nilmanifolds in $\widehat{\fM}_{2\nu}(R)$.
\end{remark}

\section{Weighted analysis on the ALG$^*$ model space}
\label{s:weighted-model}
In this section, we begin our analysis on the model space $(\widehat{\fM}_{\nu}(R), g^{\widehat{\fM}}_{\kappa_0,L})$.  To simplify notation, we will abbreviate $ (\widehat{\fM}_{\nu}(R), g^{\widehat{\fM}}_{\kappa_0,L})$ by just $(\widehat{\fM}, g^{\widehat{\fM}})$.
Without loss of generality, by scaling we can assume that the parameter $L = 1$ in Definition~\ref{d:ALG*-model}. Let us introduce the following notation
\begin{align}
\begin{split}
  E&=1, \quad E^1= V^{1/2} dx, \quad E^2= V^{1/2} dy, \quad E^3=V^{1/2} d\t2, \quad E^4=V^{-1/2}\Theta,\\ 
E^{12}&=E^1\wedge E^2, \quad E^{13}=E^1\wedge E^3, \quad \ldots \quad , \  E^{1234} = E^1 \wedge E^2 \wedge E^3 \wedge E^4.
\end{split}
\end{align}
We compute that 
\begin{align}
d E^1 = \frac{1}{2} V^{-1/2} \frac{\nu}{2\pi} \frac{1}{r} dr \wedge dx = \frac{1}{2} V^{-1/2} \frac{\nu}{2\pi} \frac{y}{r^2} dy \wedge dx = \frac{1}{2} V^{-3/2} \frac{\nu}{2\pi} \frac{y}{r^2} E^2 \wedge E^1,
\end{align}
\begin{align}
d E^2 = \frac{1}{2} V^{-1/2} \frac{\nu}{2\pi} \frac{1}{r} dr \wedge dy = \frac{1}{2} V^{-1/2} \frac{\nu}{2\pi} \frac{x}{r^2} dx \wedge dy = \frac{1}{2} V^{-3/2} \frac{\nu}{2\pi} \frac{x}{r^2} E^1 \wedge E^2,
\end{align}
\begin{align}
d E^3 = \frac{1}{2} V^{-1/2} \frac{\nu}{2\pi} \frac{1}{r} dr \wedge d \t2 = \frac{1}{2} V^{-1/2} \frac{\nu}{2\pi} \frac{x dx + y dy}{r^2} \wedge d \t2 = \frac{1}{2} V^{-3/2} \frac{\nu}{2\pi} \frac{x E^1 + y E^2}{r^2} \wedge E^3,
\end{align}
and
\begin{align}
  \begin{split}
d E^4 & = -\frac{1}{2} V^{-3/2} \frac{\nu}{2\pi} \frac{dr}{r} \wedge \Theta + V^{-1/2} \frac{\nu}{2\pi} d\t1 \wedge d\t2 \\
& = -\frac{1}{2} V^{-3/2} \frac{\nu}{2\pi} \frac{x d x + y d y}{r^2} \wedge \Theta + V^{-1/2} \frac{\nu}{2\pi} \frac{x dy - y d x}{r^2} \wedge d\t2 \\
& = -\frac{1}{2} V^{-3/2} \frac{\nu}{2\pi} \frac{x E^1 + y E^2}{r^2} \wedge E^4 + V^{-3/2} \frac{\nu}{2\pi} \frac{x E^2 - y E^1}{r^2} \wedge E^3.
\end{split}
\end{align}
By Cartan’s structural equations
\begin{align}
dE^i = - E^i_j \wedge E^j, \quad E^i_j = -E^j_i,
\end{align}
we see that 
$|E^i_j| \le C \cdot r^{-1}V^{-3/2}$.
 In other words,
$|\nabla_{E_i}E_j| \le C \cdot r^{-1}V^{-3/2}$,
 where $E_i$ are the dual orthonormal basis of $E^i$.
It follows that
$|\nabla_{E_k}\nabla_{E_i}E_j| \le C \cdot r^{-2}V^{-2}$,
 which implies that the curvature of the model space satisfies
\begin{align}\label{e:ALG*-curvature-decay}
|\mathrm{Rm}| =  O(r^{-2}V^{-2}) = O ( r^{-2} (\log r)^{-2}),
\end{align}
as $r \to \infty$. For $j \in \ZZ$ fixed, let
\begin{equation}
\label{e:jfixed}
\omega\equiv \sum\limits_{I\subset\{1,2,3,4\}} e^{\i \cdot  j \cdot \t1} \omega_{I}(r) E^I,
\end{equation} 
where $\omega_I(r)$ are smooth functions in $r$, and where the empty subset corresponds to $E = 1$. Then
\begin{equation}
\begin{split}
&\  \Big| \nabla^* \nabla \omega - \sum\limits_{I\subset\{1,2,3,4\}} (\nabla^* \nabla (e^{\i \cdot  j \cdot \t1} \omega_{I})) E^I \Big| \\
 \le & \  C \sum\limits_{I\subset\{1,2,3,4\}} \Big( |\nabla (e^{\i \cdot  j \cdot \t1} \omega_{I})| \cdot |\nabla E^I| + |\omega_{I}| \cdot |\nabla^* \nabla E^I| \Big)
\end{split}
\end{equation}
for a constant $C$ independent of $j$, where $\nabla^*$ is the $L^2$-adjoint of $\nabla$. 
By the Weitzenb\"ock formula,
\begin{align}
|\Delta \omega - \nabla^* \nabla \omega| \le C |\mathrm{Rm}| |\omega|.
\end{align}
Our convention for the Hodge star operator is that $\alpha \wedge * \beta = g(\alpha, \beta) dV_g$, the divergence operator is $\delta = - * d *$, and the Laplacian is the Hodge Laplacian $\Delta = d \delta + \delta d$.

 On the other hand, if we define the following operator
\begin{align}
L_{\RR^2,j}\omega \equiv  \sum_{I\subset\{1,2,3,4\}} e^{\i\cdot  j \cdot \t1} (\omega_{I}''+r^{-1}\cdot \omega_{I}'-j^2 \cdot r^{-2}\cdot \omega_{I}) E^I,
\end{align}
then by \eqref{Expansion-function}, 
\begin{align}
\sum\limits_{I\subset\{1,2,3,4\}} ( \nabla^* \nabla (e^{\i \cdot  j \cdot \t1} \omega_{I})) E^I = - V^{-1} L_{\dR^2,j}\omega.
\end{align}
In conclusion, there exists a constant $C$, depending only upon the model space, 
such that
\begin{align}
  |\Delta \omega + V^{-1} L_{\dR^2,j}\omega| \leq  C \cdot r^{-2} \cdot V^{-2} |\omega| + C \cdot r^{-1} \cdot V^{- 3 / 2} \cdot |\nabla\omega|,
\label{Expansion}
 \end{align}
for any $\omega$ of the form \eqref{e:jfixed}.
The estimate \eqref{Expansion} will be used to carry out the weighted analysis on the ALG$^*$ model space $\widehat{\fM}$. To start with, we define the weighted norms on the model space.
\begin{definition}
[Weighted Sobolev norms]\label{d:model-weighted-norm}
For any $\mu\in\dR$, we define the weight function
\begin{align}
\varrho_{\mu}(\bx)\equiv \fs(\bx)^{- \mu - 1},\quad \forall \ \bx\in \widehat{\fM},	
\end{align}
where $\fs(\bx)= r(\underline{\bx}) \cdot V(r(\underline{\bx}))^{\frac{1}{2}}$, $\underline{\bx} = \pr_{\dR^2}(\bx)$, $\pr_{\dR^2}:\widehat{\fM}\to \dR^2$ is the natural projection, and $r$ is the radial distance to the cone vertex of $\dR^2$. Then the Sobolev norms are defined as follows:
\begin{align}
 \|\omega\|_{L^2_{\mu}(\widehat{\fM})} \equiv \Big(\int_{\widehat{\fM}} |\omega\cdot \varrho_{\mu}|^2 \dvol_{\widehat{\fM}}\Big)^{\frac{1}{2}},\quad 
 \|\omega\|_{W^{k,2}_{\mu}(\widehat{\fM})} \equiv  \Big(\sum_{m=0}^{k} \|\nabla^m\omega\|_{L_{\mu-m}^2(\widehat{\fM})}^2\Big)^{\frac{1}{2}}, \end{align}
where $\omega$ is a tensor field on $\widehat{\fM}$. 
\end{definition}

We will require the following weighted Sobolev estimates, which hold for tensor fields of any type. 
\begin{proposition}
\label{p:weighted-Sobolev}
For any  $\mu \in \RR$ and $k\in\dN_0$, there exists a constant $C=C(\kappa_0,\nu,\mu,k)>0$ such that for any $\omega \in  W^{k + 2,2}_{\mu}(\widehat{\fM}_{\nu}(R))$,
\begin{align}\label{e:weighted-W^{2,k}}
\|\omega\|_{W^{k+2,2}_{\mu}(\widehat{\fM}_{\nu}(2R))} \le C \Big( \|\omega\|_{L^{2}_{\mu}(\widehat{\fM}_{\nu}(R))} +  \|\Delta \omega\|_{W^{k,2}_{\mu - 2}(\widehat{\fM}_{\nu}(R))} \Big).
\end{align}
\end{proposition}
\begin{proof}
The argument is standard, so we will be brief; see for example \cite[Proposition~6.16]{CVZ}. For $\bx\in\widehat{\fM}_\nu(2R)$, we consider the rescaled metric $\tilde{g}\equiv 100 \cdot d^{-2}\cdot g^{\widehat{\fM}}$, where $d$ is the $g^{\widehat{\fM}}$-distance between $\bx$ and $\{r=R\}$. It is straightforward to check that $|\Rm_{\tilde{g}}| \leq C_0$ on $B_{2}^{\tilde{g}}(\bx)$ for some constant $C_0>0$ independent of $\bx \in \widehat{\fM}_\nu(2R)$; see~\eqref{e:ALG*-curvature-decay}. The standard elliptic estimate is
\begin{align}\label{e:standard-W^{2,k}}
 \|\omega\|_{W^{k+2,2}(B_{1/2}^{\tilde{g}}(\bx))} \le C \cdot \|\omega\|_{L^{2}(B_1^{\tilde{g}}(\bx))} + C \cdot \|\Delta \omega\|_{W^{k,2}(B_1^{\tilde{g}}(\bx))}.
\end{align}
Rescaling back to $g^{\widehat{\fM}}$, and using a simple covering argument, \eqref{e:weighted-W^{2,k}} follows.
\end{proof}

\begin{proposition}
\label{p:Sobolevemb} For any $\mu \in \RR$ and $k\in\dN_0$, there exists a
 constant $C=C(\kappa_0,\nu,\mu,k)>0$ such that for any $\omega \in  W^{k + 3,2}_{\mu}(\widehat{\fM}_{\nu}(R))$, 
\begin{align}
\sum\limits_{m = 0}^k  \sup_{\bm{x} \in \widehat{\fM}_\nu(2R)} | (\mathfrak{s}(\bm{x}))^{m-\mu} \nabla^m \omega (\bm{x})|
\leq C  \Vert \omega \Vert_{ W^{k + 3,2}_{\mu}(\widehat{\fM}_\nu(R))}.\end{align}

\end{proposition}
\begin{proof}
The argument is similar to the proof of Proposition~\ref{p:weighted-Sobolev}, and is omitted. 
\end{proof}

The following estimate is key to our weighted analysis.
\begin{proposition}
	Given $\kappa_0\in\dR$, $\nu \geq 1$ and  $\mu  \in \dR\setminus \ZZ$, there exists a constant $R>0$ depending only on $\kappa_0$, $\nu$ and $\mu$ such that the following property holds. If $\omega$ is a smooth form compactly supported on a subset of $\{r>R\}$, then for any $k\in\dN_0$, there exists a constant $C=C(\kappa_0,\nu,\mu,k)>0$ such that
\begin{align}
\|\omega\|_{W^{k+2,2}_{\mu}(\widehat{\fM})} \leq C \cdot \|\Delta \omega\|_{W^{k,2}_{\mu-2}(\widehat{\fM})}.
\label{e:Weighted-analysis}
\end{align}
\label{p:weighted-estimate-model-space}
\end{proposition}
\begin{proof}
We can decompose $\omega=\sum\limits_{I\subset\{1,2,3,4\}} \omega_I E^I$ into two types. For the first type, $\omega_I$ depends only on $r$ and $\t1$. For the second type, $\int_{T^2} \omega_I(r,\t1,\t2,\t3) d \t2 d \t3 = 0$. It is easy to see that the Hodge-Laplacian $\Delta$ preserves this decomposition. So we only need to prove \eqref{e:Weighted-analysis} in two steps.

\vspace{0.1cm}
\noindent
{\bf Step 1.} We will consider the case when $\omega_I$ depends only on $r$ and $\theta_1$. In this case, we can write the coefficient function $\omega_I$ in terms of the Fourier expansion:
\begin{align}
\label{omdec}
\omega=\sum_{I\subset\{1,2,3,4\}} \sum_{j = -\infty}^{\infty} e^{\i \cdot j\cdot \t1}\omega_{I,j}(r) E^I.
\end{align}
Let us start with the proof of the following weighted $L^2$-estimate 
\begin{align}
\|\omega\|_{L_{\mu}^2(\widehat{\fM})} \leq C \cdot \|\Delta \omega\|_{L_{\mu-2}^2(\widehat{\fM})}.
\label{e:L^2-estimate}
\end{align}
Since $\Delta$ preserves the decomposition in \eqref{omdec}, and the subspaces for different $j$ are $L^2_{\mu}$-orthogonal, we can assume that
\begin{align}
\omega=\sum_{I\subset\{1,2,3,4\}} e^{\i  \cdot j \cdot \t1}\omega_{I}(r) E^I  \label{e:frequency-j-omega}
\end{align}
for some fixed $j \in \ZZ$. Then we need to prove \eqref{e:L^2-estimate} for a constant $C$ independent of $j$. The proof of \eqref{e:L^2-estimate} will be reduced to computations in $\dR^2$. We will need the following two claims.

\vspace{0.1cm}
\noindent
{\bf Claim 1.}
Let $\alpha\neq -1$, $\beta\in\dR$ and let $R>0$ satisfy 
\begin{equation}
\Big|\beta \cdot \frac{\nu}{2\pi} \cdot (\kappa_0 + \frac{\nu}{2\pi}\log R)^{-1}\Big| \leq \frac{|\alpha+1|}{2}.
\end{equation} Then for any $f\in C_0^{\infty}(\{r\geq R\})$,
\begin{align}
\int_R^{\infty}|f|^2 r^{\alpha}V^{\beta}dr
\leq \frac{16}{(\alpha+1)^2}\int_R^{\infty}|f'|^2r^{\alpha+2}V^{\beta}dr.
\end{align}
\begin{proof}
  It is straightforward that if $R$ is chosen such that
  \begin{align}
    \Big|\beta \cdot \frac{\nu}{2\pi} \cdot (\kappa_0 + \frac{\nu}{2\pi} \log R)^{-1}\Big| \leq \frac{|\alpha+1|}{2},
\end{align}
then for any $r\geq R$,
\begin{align}
\Big|\frac{d}{dr}(r^{\alpha+1}V^{\beta})\Big|	
= r^{\alpha}V^{\beta}\cdot 
\Big|\alpha + 1 + \beta \cdot \frac{\nu}{2\pi} \cdot V^{-1}\Big|\geq   \frac{|\alpha+1|}{2} \cdot r^{\alpha}V^{\beta}.
\end{align}
So it follows that 
\begin{align}
\begin{split}
  \int_{R}^{\infty}|f|^2 r^{\alpha}V^{\beta} dr 
& \leq \frac{2}{|\alpha + 1|} \Big|\int_R^{\infty} |f|^2 d(r^{\alpha+1}V^{\beta})\Big|
\\
& =  \frac{4}{|\alpha + 1|} \Big|\int_R^{\infty} \Rea(f \cdot \bar f') \cdot r^{\alpha+1}V^{\beta}dr\Big|
\\
& \leq \frac{4}{|\alpha + 1|}\Big(\int_R^{\infty}|f|^2 r^{\alpha}V^{\beta}dr\Big)^{\frac{1}{2}}\Big(\int_R^{\infty}|f'|^2r^{\alpha+2}V^{\beta}dr\Big)^{\frac{1}{2}}.
\end{split}
\end{align}
Then the desired inequality immediately follows.
\end{proof}

\vspace{0.1cm}
\noindent
{\bf Claim 2.}  Let $\omega$ be defined in \eqref{e:frequency-j-omega}.
 Then for $\mu\in\dR\setminus \dZ$ and $R>0$ satisfying 
\begin{equation}
\Big|-\mu \cdot \frac{\nu}{\pi} \cdot (\kappa_0 + \frac{\nu}{2\pi}\log R)^{-1}\Big| \leq \min\{|j+\mu|,|j-\mu|\},
\end{equation}
we have that
\begin{align}
\|\omega\|_{L^2_{\mu}(\widehat{\fM})}^2 \leq 16 \cdot (j+\mu)^{-2} \cdot (j-\mu)^{-2} \cdot \|r^2\cdot L_{\RR^2,j}\omega\|_{L^2_{\mu}(\widehat{\fM})}^2.
\end{align}

\begin{proof}
 Recall that $\dvol_{\widehat{\fM}} = V \cdot r \cdot dr \wedge d\theta_1\wedge d\theta_2 \wedge \Theta$,
and thus \begin{align}
 \|\omega\|^2_{L^2_{\mu}(\widehat{\fM})}  = 8 \pi^3 \sum_{I\subset\{1,2,3,4\}} \int_{R}^{\infty} |\omega_{I}|^2 r^{-2\mu-1}V^{-\mu} dr.
 \end{align}
Then let us estimate the weighted $L^2$-integral of $r^2 \cdot L_{\dR^2,j}\omega$. By definition,
\begin{align}
\begin{split}
\|r^2 L_{\RR^2,j}\omega\|_{L^2_{\mu}(\widehat{\fM})}^2 & = 8 \pi^3 \sum_{I\subset\{1,2,3,4\}} \int_{R}^{\infty} |\omega_{I}''+r^{-1}\cdot\omega_{I}'- j^2 \cdot r^{-2} \cdot \omega_{I}|^2 r^{-2\mu+3} V^{-\mu}dr \\
& = 8 \pi^3 \sum_{I\subset\{1,2,3,4\}} \int_{R}^{\infty} | ( r^{2j+1} ( r^{-j} \omega_{I})' )'|^2 r^{-2\mu-2j+1}V^{-\mu}dr.\end{split} \end{align}
Since $\omega$ has compact support in $\widehat{\fM}$, applying Claim 1 to above integral,
\begin{align}
\begin{split}
\|r^2 L_{\RR^2,j}\omega\|_{L^2_{\mu}(\widehat{\fM})}^2 & \geq 8 \pi^3 \cdot \frac{(-2\mu-2j)^2}{16}  \sum_{I\subset\{1,2,3,4\}} \int_{R}^{\infty} |r^{2j+1} ( r^{-j} \omega_{I})'|^2 r^{-2\mu-2j-1} V^{-\mu}dr
\\
& = 2 \pi^3 (j+\mu)^2 \sum_{I\subset\{1,2,3,4\}} \int_{R}^{\infty} |( r^{-j} \omega_{I})'|^2 r^{-2\mu+2j+1} V^{-\mu}dr \\
& \geq 2 \pi^3 (j+\mu)^2 \cdot \frac{(-2\mu+2j)^2}{16} \sum_{I\subset\{1,2,3,4\}} \int_{R}^{\infty} |r^{-j} \omega_{I}|^2 r^{-2\mu+2j-1} V^{-\mu}dr \\
& = \frac{\pi^3}{2} (j+\mu)^2 (j-\mu)^2 \sum_{I\subset\{1,2,3,4\}} \int_{R}^{\infty} |\omega_{I}|^2 r^{-2\mu-1}  V^{-\mu}dr  \\
& = \frac{(j+\mu)^2 (j-\mu)^2}{16} \|\omega\|_{L^2_{\mu}(\widehat{\fM})}^2.
\end{split}
\end{align}
The proof of the claim is done.
\end{proof}
Now we are ready to finish the proof of \eqref{e:L^2-estimate}. Recall that by \eqref{Expansion},
\begin{align}
|r^2 V \Delta \omega + r^2 L_{\RR^2,j}\omega| \leq  C \cdot V^{-1} |\omega| + C \cdot r \cdot V^{- 1 / 2} \cdot |\nabla\omega| \le C \cdot V^{-1} (|\omega| + r \cdot V^{1 / 2} \cdot |\nabla\omega|)
\end{align}
for $r \ge R$, where $C>0$ is a constant independent of $j$. By Claim 2,
\begin{align}
\|\omega\|_{L^2_{\mu}(\widehat{\fM})} \leq  C \cdot \|r^2 L_{\RR^2,j}\omega\|_{L^2_{\mu}(\widehat{\fM})}.
\end{align}
Consequently, 
\begin{align}\label{e:absorbing-L2-omega}
  \begin{split}
\|\omega\|_{L^2_{\mu}(\widehat{\fM})} & \le C \cdot \|\fs^2 \cdot \Delta \omega\|_{L^2_{\mu}(\widehat{\fM})} + C \cdot V(R)^{-1} \|\omega\|_{W^{1,2}_{\mu}(\widehat{\fM})} \\
& =  C \cdot \|\Delta \omega\|_{L^2_{\mu - 2}(\widehat{\fM})} + C \cdot V(R)^{-1} \|\omega\|_{W^{1,2}_{\mu}(\widehat{\fM})} \\
& \le C \cdot \|\Delta \omega\|_{L^2_{\mu - 2}(\widehat{\fM})} + C \cdot V(R)^{-1} \|\omega\|_{W^{2,2}_{\mu}(\widehat{\fM})}.
\end{split}
\end{align}
If $R$ is chosen sufficiently large,  \eqref{e:L^2-estimate} follows by  plugging \eqref{e:weighted-W^{2,k}} for $k =0$ into \eqref{e:absorbing-L2-omega}.

The estimate \eqref{e:Weighted-analysis} follows immediately from  \eqref{e:L^2-estimate} and \eqref{e:weighted-W^{2,k}}.
  
\vspace{0.1cm}
\noindent
{\bf Step 2.} We will
 consider the case when $\omega = \sum_{I\subset\{1,2,3,4\}}\omega_I(r,\theta_1,\theta_2,\theta_3) \cdot E^I$ satisfies 
\begin{align}\int_{T^2} \omega_I(r,\t1,\t2,\t3) d \t2 d \t3 = 0.
\label{e:integral-over-torus-vanishing}
\end{align}
We will prove that there exists $C>0$ such that if $\omega$ satisfies \eqref{e:integral-over-torus-vanishing}, then
 \begin{align}
 \|\omega\|_{W^{2,2}_{\mu}(\widehat{\fM})} \leq C \cdot \|\Delta \omega\|_{L^2_{\mu - 2}(\widehat{\fM})}.	\label{e:weighted-W^{2,2}}
 \end{align}
The higher order estimate \eqref{e:L^2-estimate} then follows from Proposition~\ref{p:weighted-Sobolev}.

 First, we will show that there is a constant $C>0$ such that if the coefficient function of $\omega$ satisfies \eqref{e:integral-over-torus-vanishing}, then for any sufficiently large $R$ and for any $r\in[R,\infty)$,
 \begin{align}
 \int_{\mathcal{A}_{r/2,2r}}|\omega|^2 \dvol_{\widehat{\fM}} \leq C\cdot  V(r) \cdot \int_{\mathcal{A}_{r/2,2r}} |\nabla\omega|^2\dvol_{\widehat{\fM}}, \label{e:poincare-inequality}
 \end{align}
where $\mathcal{A}_{r/2,2r} \equiv \pr_{\RR^2}^{-1}(A_{r/2,2r}(0^2))$.
There are two sub-cases to analyze: 
\begin{enumerate}
\item $\omega_I$ satisfies $\int_{S^1} \omega_I(r,\t1,\t2,\t3) d \t3 = 0$ for all $r$, $\t1$, and $\t2$,  
\item $\omega_I=\omega_I(r, \t1, \t2)$  and it satisfies $\int_{S^1} \omega_I(r,\t1,\t2) d \t2 = 0$ for all $r$ and $\t1$.
\end{enumerate}
In the first case, applying the Poincar\'e inequality on the circle parametrized by $\theta_3$, we find that for any $(r,\theta_1,\theta_2)\in[R,\infty)\times[0,\pi]\times[0,2\pi]$, \begin{align} \int_{S^1}|\omega_I|^2 d\t3 \leq C \int_{S^1} |\partial_{\t3}\omega_I |^2 d\t3  \leq \frac{C}{V(r)} \int_{S^1} |\nabla\omega|^2(r, \t1, \t2, \t3) d \t3.
\end{align}
 In the second case, applying Poincar\'e inequality on  the circle parametrized by $\theta_2$, it follows that for any $(r,\theta_1)\in[R,\infty)\times[0,\pi]$,
\begin{align}
 \int_{S^1} |\omega_I|^2 d \t2 \leq C \int_{S^1}|\partial_{\t2}\omega_I|^2 d\t2\leq   C \cdot V(r) \cdot  \int_{S^1} |\nabla\omega|^2(r, \t1,\t2) d\t2.
\end{align}
 Combining the above two cases and integrating over $\widehat{\fM}$, inequality \eqref{e:poincare-inequality}
 immediately follows from the condition $R\gg1$.

Now we proceed to prove \eqref{e:weighted-W^{2,2}}. Let $r_i = 2^{i+1} \cdot R$. Then $\{A_{r_i/2,2r_i}(0^2)\}_{i=0}^{\infty}$ is a covering of $\dR^2\setminus \overline{B_R(0^2)}$ consisting of a sequence of annuli such that the number of overlaps at every point is bounded by $2$. We denote $\mathcal{A}_i\equiv \pr^{-1}_{\RR^2}(A_{r_i/2,2r_i}(0^2))$ and $\fs_i \equiv \fs(r_i)$.
Let $\{\chi_i\}_{i=0}^{\infty}$ be a partition of unity subordinate to the above covering such that $\Supp(\chi_i)\subset  A_{r_i/2,2r_i}(0^2)$
 and 
\begin{equation}
|\nabla_{\mathbb{R}^2}\chi_i|^2+|\nabla^2_{\mathbb{R}^2}\chi_i| \leq C \cdot r_i^{-2}
\end{equation}
on $\dR^2$.
 We still denote by $\chi_i$ the lifting of $\chi_i$ to $\widehat{\fM}$. In terms of the model metric $g^{\widehat{\fM}}$, the following estimate holds on $\widehat{\fM}$:
\begin{align} |\nabla\chi_i|^2 + |\nabla^2\chi_i|\leq C \cdot \fs_i^{-2},\end{align}
where $C>0$ is independent of $i$.
By \eqref{e:poincare-inequality},
\begin{align}
	\int_{\widehat{\fM}}|\chi_i \omega|^2\leq C \cdot V_i\cdot  \int_{\widehat{\fM}}|\nabla(\chi_i\omega)|^2,
\end{align}
where $V_i\equiv V(r_i)$.
Since $\chi_i \omega$ has compact support in $\widehat{\fM}$, integrating by parts and applying the Cauchy-Schwartz inequality, we have that
\begin{align}
\begin{split}
\int_{\widehat{\fM}}|\chi_i \omega|^2\dvol_{\widehat{\fM}} & \leq C \cdot  V_i \cdot  \int_{\widehat{\fM}} (\chi_i \omega) \cdot \nabla^*\nabla (\chi_i \omega)\dvol_{\widehat{\fM}}
\nonumber\\
& \leq C \cdot V_i \cdot  \Big(\int_{\widehat{\fM}} |\chi_i \omega|^2\dvol_{\widehat{\fM}}\Big)^{\frac{1}{2}}\Big(\int_{\widehat{\fM}}|\nabla^*\nabla (\chi_i \omega)|^2\dvol_{\widehat{\fM}}\Big)^{\frac{1}{2}}.
\end{split}
\end{align}
So it follows that
\begin{align}
\int_{\widehat{\fM}}|\chi_i \omega|^2\dvol_{\widehat{\fM}} \leq C \cdot V_i^2 \cdot \int_{\widehat{\fM}}|\nabla^*\nabla (\chi_i \omega)|^2\dvol_{\widehat{\fM}}.
\end{align}
Using the Weitzenb\"ock formula, we have that
\begin{align}
\begin{split}
  |\nabla^* \nabla(\chi_i \omega)|
 & \leq |\Delta(\chi_i\omega)|+ C\cdot |\Rm|\cdot |\chi_i\cdot\omega|
\\
& \leq  |\Delta(\chi_i\omega)|+ C\cdot \fs_i^{-2} \cdot V^{-1} \cdot  |\omega|
\\
& \leq |\Delta\omega| + C\cdot \fs_i^{-1}\cdot |\nabla\omega| +  C\cdot \fs_i^{-2} \cdot |\omega|.
\end{split}
\end{align}
Therefore,
\begin{align}
  \begin{split}
\int_{\widehat{\fM}} |\chi_i \omega|^2 \dvol_{\widehat{\fM}}
& \leq C \cdot V_i^2 \int_{\widehat{\fM}} |\nabla^*\nabla (\chi_i \omega)|^2 \dvol_{\widehat{\fM}} \\
 &\leq \frac{C}{r_i^4} \int_{\mathcal{A}_i}\Big(|\Delta \omega|^2 \cdot \fs_i^4 + |\nabla \omega|^2 \cdot \fs_i^2 +   |\omega|^2 \Big)\dvol_{\widehat{\fM}}.
\end{split}
\end{align}  Taking the weighted $L^2$-norm, we find that
\begin{align}
\label{e:Noninvariant-L2}
\begin{split}
  \|\omega\|_{L_{\mu}^2(\widehat{\fM})}^2
& = \int_{\widehat{\fM}}|\omega\cdot \varrho_{\mu}|^2\dvol_{\widehat{\fM}}
\\
& \leq C \cdot \sum\limits_{i} \fs_i^{-2\mu-2}\cdot \int_{\widehat{\fM}}|\chi_i\omega|^2 \dvol_{\widehat{\fM}}
\\
& \leq C\cdot \sum\limits_{i} \frac{1}{r_i^4}\int_{\mathcal{A}_i}(|(\Delta\omega)\cdot\varrho_{\mu-2}|^2 + |(\nabla\omega)\cdot\varrho_{\mu-1}|^2 + |\omega \cdot\varrho_{\mu}|^2)\dvol_{\widehat{\fM}}
\\
& \leq \frac{C}{R^4} \cdot (\|\Delta\omega\|_{L_{\mu-2}^2(\widehat{\fM})}^2 + \|\nabla\omega\|_{L_{\mu-1}^2(\widehat{\fM})}^2 + \|\omega\|_{L_{\mu}^2(\widehat{\fM})}^2),
\end{split}
\end{align}
where the last inequality follows since the number of overlaps is bounded by $2$. 
It follows that
\begin{align}
\label{e:abc}
	 \|\omega\|_{L^{2}_{\mu}(\widehat{\fM})} \leq C \cdot  \|\Delta \omega\|_{L_{\mu-2}^2(\widehat{\fM})} + \frac{C}{R^2}(\|\nabla\omega\|_{L_{\mu-1}^2(\widehat{\fM})} + \|\omega\|_{L_{\mu}^2(\widehat{\fM})}).
\end{align}
If $R$ is chosen sufficiently large, then \eqref{e:weighted-W^{2,2}}
follows from \eqref{e:abc} and Proposition \ref{p:weighted-Sobolev}. 

To finish the proof, we can write any form $\omega = \omega_1 + \omega_2$ as the sum of these two types of forms. We then have
\begin{align}
\label{e:w222}
\begin{split}
\|\omega\|_{W_{\mu}^{2,2}(\widehat{\fM})} 
&\leq \|\omega_1\|_{W_{\mu}^{2,2}(\widehat{\fM})} + \|\omega_2\|_{W_{\mu}^{2,2}(\widehat{\fM})} \\
&\leq C \Big( \|\Delta\omega_1\|_{L_{\mu-2}^2(\widehat{\fM})}
+ \|\Delta\omega_2\|_{L_{\mu-2}^2(\widehat{\fM})}\Big) = C  \|\Delta\omega\|_{L_{\mu-2}^2(\widehat{\fM})},
\end{split}
\end{align}
where the last equality follows since the two subspaces are orthogonal in $L^2_{\mu-2}(\widehat{\fM})$. The higher order estimates then follow from \eqref{e:w222} and Proposition~\ref{p:weighted-Sobolev}. 
\end{proof}

\begin{corollary}
Given any $\mu\in\dR\setminus \dZ$, there exists a constant $\epsilon_0>0$ such that the following property holds. Let $\omega=\sum\limits_{I\subset\{1,2,3,4\}}  \omega_{I} E^I$ be a smooth form on $\widehat{\fM}$ that satisfies $\Delta \omega=0$ and $\|\omega\|_{W^{2,2}_{\mu}(\widehat{\fM})}<\infty$. If 
\begin{equation}
\int_{T^2} \omega_I(r,\t1,\t2,\t3) d \t2 d \t3 = 0
\end{equation}
for every $I\subset\{1,2,3,4\}$, then $|\omega|=O(e^{-\epsilon_0\cdot r})$ as $r\to\infty$.
\label{Exponential-decay}
\end{corollary}

\begin{proof}
Let $R_0$ be the constant $R$ in Proposition \ref{p:weighted-estimate-model-space} depending on $\kappa_0$, $\nu$, and $\mu$. If $\eta$ is a smooth form on $\widehat{\fM}$ with compact support on a subset of $\{r>R_0\}$, \eqref{e:weighted-W^{2,2}} and \eqref{e:Noninvariant-L2} imply that
\begin{equation}
\label{e:Noninvariant-improved-L2}
\|\eta\|_{L_{\mu}^2(\widehat{\fM})}^2 \leq \frac{C}{R^4} \cdot \|\Delta\eta\|_{L_{\mu-2}^2(\widehat{\fM})}^2.
\end{equation}
For any $R>R_0$, let $\varphi_R$ be a cut-off function on $\dR^2$ such that
 \begin{align}
 \varphi_R =
 \begin{cases}
 	1, & r\geq R + 2,
 	\\
 	0, & r\leq R+1,
 \end{cases}	
 \end{align}
and $|\nabla \varphi_R|_{\dR^2}+ |\nabla^2 \varphi_R|_{\dR^2} \leq C$. We still denote by $\varphi_R$ the lifting of $\varphi_R$ to the model space $\widehat{\fM}$. 
Then it holds that
$|\nabla \varphi_R|_{\widehat{\fM}}^2 + |\nabla^2\varphi_R|_{\widehat{\fM}}\leq C\cdot V^{-1}$.	
 Since $\varphi_R \omega$ can be approximated by compactly supported smooth forms in the $W^{k,2}_\mu$-norm, \eqref{e:Noninvariant-improved-L2} also holds for $\varphi_R\omega$.
Using \eqref{e:Noninvariant-improved-L2}, we estimate
\begin{equation}
\begin{split}
\int_{r \geq   R + 3} |\omega\cdot  \varrho_{\mu}|^2 \dvol_{\widehat{\fM}} &\leq \frac{C}{R^4} \cdot  \int_{r \geq   R + 1}|\Delta (\varphi_R \omega) \cdot \varrho_{\mu-2}|^2\dvol_{\widehat{\fM}}  \\
&\leq C R^{-2\mu-2} V^{-\mu+1} \cdot \int_{R + 1 \leq  r \leq  R + 2} |\Delta (\varphi_R \omega)|^2 \dvol_{\widehat{\fM}}
\\
&\leq C R^{-2\mu-2} V^{-\mu-1} \cdot\int_{R + 1 \leq  r \leq  R + 2} ( |\omega|^2 + |\nabla \omega|^2 \cdot V) \dvol_{\widehat{\fM}}
\\
&\leq C R^{-2\mu-2} V^{-\mu-1} \cdot\int_{R \leq  r \leq  R + 3} |\omega|^2 \dvol_{\widehat{\fM}}
\\
&\leq C_0\cdot\int_{R \leq  r < R + 3}  |\omega|^2\cdot \varrho_{\mu}^2 \dvol_{\widehat{\fM}}.
\end{split}
\end{equation}
Notice that the fourth inequality follows from Proposition \ref{p:weighted-Sobolev}. Therefore, \begin{equation}\int_{r \geq   R + 3} |\omega\cdot  \varrho_{\mu}|^2 \dvol_{\widehat{\fM}}  \leq \frac{C_0}{C_0+1} \cdot \int_{r \geq   R} |\omega\cdot  \varrho_{\mu}|^2 \dvol_{\widehat{\fM}},
\end{equation}
which implies that $\int_{r\geq R} |\omega \cdot \varrho_{\mu}|^2 \dvol_{\widehat{\fM}} = O(e^{-\epsilon_0\cdot R})$ for some $\epsilon_0>0$ as $R \to \infty$. Therefore, $\omega$ also decays exponentially by Proposition \ref{p:weighted-Sobolev} and Proposition \ref{p:Sobolevemb}. 
\end{proof}

\section{Weighted analysis on ALG$^*$ manifolds}
\label{s:weighted}
We next transfer the previous estimates on the model space to any ALG$^*$ manifold. Without loss of generality, by scaling we can assume that the parameter $L = 1$ in Definition~\ref{d:ALGstar}.

\begin{definition}
[Weighted Sobolev norms] Let $(X, g)$ be an $\ALG^*_{\nu}$-$\Gamma$ manifold for $\nu \in \ZZ_+$ together with a diffeomorphism 
$\Phi: \widehat{\fM}_{\nu}(R) /\Gamma \longrightarrow X \setminus X_R$.	
 For any fixed parameters $\kappa_0\in\dR$, and $\mu\in\dR$, we define the weight function $\hat{\varrho}_{\mu}$ on $X$,
 \begin{align}
\hat{\varrho}_{\mu}(\bx)\equiv
\begin{cases}
	1, & r\leq 2R,
	\\
	\varrho_{\mu}(\Phi^{-1}(\bx)), & r\geq 3R,
\end{cases}	
\end{align}
where $R$ is the radius as in Definition \ref{d:ALGstar}, and $\varrho_{\mu}$ is the weight function on $\fM = \widehat{\fM} / \Gamma$ as in Definition \ref{d:model-weighted-norm}. 
Then the weighted Sobolev norms are defined as follows:
\begin{align}
 \|\omega\|_{L^2_{\mu}(X)} \equiv \Big(\int_{X} |\omega\cdot \hat{\varrho}_{\mu}|^2 \dvol_{X}\Big)^{\frac{1}{2}},\quad 
 \|\omega\|_{W^{k,2}_{\mu}(X)} \equiv  \Big(\sum_{m=0}^{k} \|\nabla^m\omega\|_{L_{\mu-m}^2(X)}^2\Big)^{\frac{1}{2}}.\end{align}
\end{definition}

\begin{remark}Notice that Proposition \ref{p:weighted-Sobolev} and Proposition \ref{p:Sobolevemb} are stated on the ALG$^*$ model space $\widehat{\fM}$. Since the weighted norms on an ALG$^*$ gravitational instanton $X$ and its model space $\widehat{\fM}$ differ by a uniform multiplicative constant, in our applications, we may also quote these two propositions for ALG$^*$ manifolds. 
\end{remark}

We need the following notations on an ALG$^*$ manifold $(X, g)$ and its asymptotic model $(\fM, g^{\fM})$. 
\begin{definition} Let $\mathcal{H}_{\mu}^p(X)$ be the space of all smooth $p$-forms $\omega$ on $(X, g)$ that satisfy $\Delta_{X}\omega=0$  and $\|\omega\|_{L^2_{\mu}(X)}<\infty$, where $\Delta_{X}$ is the Hodge Laplacian on $(X, g)$.  \end{definition}

\begin{definition}
\label{d:Z}
 Given $p\in\{0,1,2\}$ and $\fq\in\dZ$, let $\mathcal{Z}_{\fq}^p(\widehat{\fM})$ be the linear space of $p$-forms with a basis $\{u_{\fq,i}\cdot e^{\i \cdot m_{\fq,i} \cdot \t1}:1\leq i\leq \dim\mathcal{Z}_{\fq}^p(\widehat{\fM})\}$,
where for each $i$,
\begin{align}
	u_{\fq,i} = \sum\limits_I \omega_{I, i}(r) E^I, 
\end{align} $\omega_{I,r}(r)$ is a radial function,
and $I\subset \{1,2,3,4\}$ satisfies $|I| = p$.
For any $p\in\{0,1,2\}$ and $\fq\in\dZ$, the linear space $\mathcal{Z}_{\fq}^p(\widehat{\fM})$ is characterized as follows.
\begin{itemize}
\item When $p=0$, the basis is defined in Lemma~\ref{l:Z0}
\item When $p=1$, the basis of $\mathcal{Z}_{\fq}^1(\widehat{\fM}) \equiv \mathcal{Z}_{\fq}^{1,\I}(\widehat{\fM}) \oplus \mathcal{Z}_{\fq}^{1,\II}(\widehat{\fM})$ is defined in Lemma~\ref{l:Z1I} and Lemma~\ref{l:Z1II}
\item When $p=2$, the basis of $\mathcal{Z}_{\fq}^2(\widehat{\fM}) \equiv \mathcal{Z}_{\fq}^{2,+}(\widehat{\fM}) \oplus \mathcal{Z}_{\fq}^{2,-}(\widehat{\fM})$ is defined in \eqref{e:Z2+} and Lemma~\ref{l:Z2-}.
\end{itemize}

\end{definition}

Our main result in this section is the following Fredholm package for the Hodge Laplacian on $\ALG^*$ manifolds with respect to the weighted Sobolev norms. 
\begin{proposition}
\label{Weighted-analysis-X}
Let $(X, g)$ be an $\ALG^*_{\nu}$-$\Gamma$ manifold of order $\mathfrak{n} > 0$. Then the following properties hold.
\begin{enumerate}

\item  For any $\mu  \in \dR\setminus\ZZ$ and $k\in\dN_0$, there exist constants $R(X,\mu)>0$ and $C(X, \mu, k)>0$ such that for any $p$-form $\omega\in W^{k+2,2}_\mu(X)$,
\begin{align}
\|\omega\|_{W^{k+2,2}_{\mu}(X)} \leq C \cdot (\| \Delta_X \omega\|_{W^{k,2}_{\mu - 2}(X)} + \|\omega\|_{L^2(\{r \le 3R\}\subset X)}).
\label{e:Weighted-analysis-X}
\end{align}

\item  For any $\mu  \in\dR\setminus \ZZ$ and $k\in\dN_0$, the operator 
\begin{equation}
\Delta_{X}: W^{k+2,2}_\mu (X) \to W^{k,2}_{\mu-2}(X)
\end{equation}
is a Fredholm operator.
 Then for any $p$-form $\omega\in W^{k,2}_{\mu-2}(X)$, 
\begin{equation}
\Delta_{X} \xi=\omega
\end{equation} has a solution $\xi\in W^{k+2,2}_\mu(X)$ if and only if for all $\eta\in \mathcal{H}_{-\mu}^p(X)$,
\begin{align}
\int_{X} (\omega,\eta) \dvol_g = 0.
\end{align}

\item For any $\mu  \in\dR\setminus \ZZ$, $k\in\dN_0$, and $p$-form $\omega\in  W^{k,2}_{\mu}(X)$, there exists some $\xi\in W^{k+2,2}_{\mu+2}(X)$ such that $\Delta_{X} \xi=\omega$ when $r\ge 2R$.

\item  Let $\mu  \in \dR\setminus \ZZ$ and $\delta\in(0,\min\{1,\mathfrak{n}\})$. Consider any form $\omega\in W^{k,2}_{\mu}(X)$ such that $\Delta_{X}\omega=0$ when $r\ge 2R$. If $\dZ\cap [\mu-\delta,\mu]=\emptyset$, then 
$\omega \in W^{k,2}_{\mu-\delta}(X)$.
If there is some $\fq \in \dZ\cap (\mu-\delta,\mu)$, then the pull back of $\omega$ to $\widehat{\fM}$ can be written as the sum of an element in $\mathcal{Z}_{\fq}^p(\widehat{\fM})$ and an element in $W^{k,2}_{\mu-\delta}(\widehat{\fM})$.
\end{enumerate}
\end{proposition}
  
\begin{proof} Recall that, by definition, the $\ALG^*_{\nu}$-$\Gamma$ manifold $(X,g)$ is equipped with a diffeomorphism $\Phi:\widehat{\fM}_{\nu}(R) /\Gamma \rightarrow X \setminus X_R$ such that  $|\Phi^*g - g^\fM | =O(\fs^{-\mathfrak{n}})$. For the simplicity of notation, we do not distinguish the function $\fs$ defined on $\widehat{\fM}_{\nu}(R)$  with the function $\fs\circ\Phi^{-1}$ defined on $X\setminus X_R$.

For (1), for any $p$-form $\omega$ with compact support in $\{r>R\}\subset \fM$, we have that
\begin{align}
(1+O(\fs(R)^{-\mathfrak{n}}))^{-1} \|\omega\|_{W^{k+2,2}_{\mu}(\fM)} \leq \|\omega\|_{W^{k+2,2}_{\mu}(X)}\leq (1+O(\fs(R)^{-\mathfrak{n}})) \|\omega\|_{W^{k+2,2}_{\mu}(\fM)}.
\end{align}
Then we will not distinguish $W^{k+2,2}_{\mu}(\fM)$ with $W^{k+2,2}_{\mu}(X)$.
Moreover,
\begin{align}
\|\Delta_{X}\omega\|_{W^{k,2}_{\mu-2}(X)}-\| \Delta_\fM \omega\|_{W^{k,2}_{\mu-2}(X)} \leq C \cdot \fs(R)^{-\mathfrak{n}}\cdot \|\omega\|_{W^{k+2,2}_{\mu}(X)}.
\end{align}
We also identify a form on $\fM$ with a $\Gamma$-invariant form on $\widehat{\fM}$.
By \eqref{e:Weighted-analysis}, 
 for $R$ sufficiently large, we have
\begin{equation}
\|\omega\|_{W^{k+2,2}_\mu} \leq C \|\Delta_{X} \omega\|_{W^{k,2}_{\mu - 2}}.
\end{equation}
 Let $\chi$ be a cut-off function on $\RR^2$ with $\Supp(\chi)\subset B_{2R}$ and $\chi \equiv 1$ in $B_R$. If $\omega\in W^{k+2,2}_\mu(X)$, then $(1-\chi)\omega$ can be approximated by smooth forms with compact support in $\{r>R\}$. Therefore,
\begin{align}
\begin{split}
\|\omega\|_{W^{k+2,2}_\mu(X)} & \leq \|\chi\omega\|_{W^{k+2,2}_\mu(X)} + \|(1-\chi)\omega\|_{W^{k+2,2}_\mu(X)} \\
&\leq C (\|\omega\|_{W^{k+2,2}_\mu(\{r \le 2R\}\subset X)} + \|\Delta_{X} ((1-\chi)\omega)\|_{W^{k,2}_{\mu-2}(X)} ) \\
&\leq C (\|\omega\|_{L^2(\{r \le 3R\}\subset X)} + \|\Delta_{X} \omega\|_{W^{k,2}(\{r \le 3R\}\subset X)} + \|\Delta_{X} \omega\|_{W^{k,2}_{\mu-2}(X)} ) \\
&\leq C (\|\omega\|_{L^2(\{r \le 3R\}\subset X)} + \|\Delta_{X} \omega\|_{W^{k,2}_{\mu-2}(X)}).
\end{split}
\end{align}

For (2), it is straightforward to check that for any $k\in\dN_0$ and any $p$-form $\omega\in W^{k+2,2}_\mu(X)$, we have
\begin{equation}
\|\Delta_{X} \omega\|_{W^{k,2}_{\mu-2}(X)} \leq C \cdot \|\omega\|_{W^{k+2,2}_\mu(X)}.
\end{equation}

First, to prove $\dim(\ker(\Delta_{X}))<\infty$, let us take any sequence $\omega_j\in \ker(\Delta_{X})$ with \begin{equation}
\|\omega_j\|_{W^{k+2,2}_\mu(X)}=1.
\end{equation}
Rellich's theorem implies that the inclusion 
\begin{equation}
W^{k+2,2}_\mu(X)\hookrightarrow L^2(\{r \le 3R\}\subset X)
\end{equation} is compact. Combining this and \eqref{e:Weighted-analysis-X}, a subsequence of $\omega_j$ converges to $\omega_{\infty}$ in the $W^{k+2,2}_\mu$-norm. Then the unit sphere of $\ker(\Delta_{X})$ is compact, and hence 
\begin{equation}
\dim(\ker(\Delta_{X}))<\infty.
\end{equation}
Moreover, standard elliptic estimate implies 
\begin{equation}
\ker(\Delta_{X})=\mathcal{H}^{p}_{\mu}(X).
\end{equation}

Next, we will show 
$\Image(\Delta_{X})$ is closed in the $W^{k,2}_{\mu-2}$-norm. This follows from the claim that for any $p$-form $\omega\in (\mathcal{H}_{\mu}^{p}(X))^{\perp}\subset W^{k+2,2}_{\mu}(X)$ in terms of the $W^{k+2,2}_\mu$-inner product, we have
\begin{align}
\|\omega\|_{W^{k+2,2}_\mu(X)} \leq C \| \Delta_{X} \omega\|_{W^{k,2}_{\mu-2}(X)}.
\label{e:Weighted-analysis-mod-kernel}
\end{align}
Suppose not, then there exists a sequence of $\omega_j\in (\mathcal{H}_{\mu}^{p}(X))^{\perp}$ such that 
\begin{align}
\|\omega_j\|_{W^{k+2,2}_\mu(X)}=1  \quad \text{and} \quad  
 \|\Delta_{X} \omega_j\|_{W^{k,2}_{\mu}(X)}\to 0.
\end{align}
Combining Rellich's theorem and \eqref{e:Weighted-analysis-X}, passing to a subsequence, we have a limit 
\begin{equation}
\omega_{\infty}\in \mathcal{H}_{\mu}^p(X)\subset  W^{k+2,2}_\mu(X).
\end{equation}
Since $\omega_{\infty}\in (\mathcal{H}_{\mu}^p(X))^{\perp}$, we have $\omega_{\infty} = 0$.

Now we show the solvability criterion. This implies that $\dim(\coker(\Delta_{X}))<\infty$. When $k=0$, we can characterize $\Image(\Delta_{X})$ in $L^2_{\mu-2}(X)$ by finding its orthogonal complement in terms of the $L^2_{\mu-2}(X)$-inner product. Notice that $\tilde\eta\perp \Image(\Delta_{X})$ if and only if
\begin{align}
\int_{X} (\Delta_{X} \xi, \tilde\eta) \cdot (\hat{\varrho}_{\mu-2})^2\dvol_g = 0\quad \forall \xi \in W^{2,2}_\mu(X).
\end{align}
This is identical to 
\begin{equation}
\Delta_{X}((\hat{\varrho}_{\mu-2})^2\cdot \tilde\eta)=0
\end{equation}
 in the distributional sense. By standard elliptic regularity, this also coincides with the condition that
\begin{equation}
\eta \equiv (\hat{\varrho}_{\mu-2})^2\cdot  \tilde\eta \in \mathcal{H}_{-\mu}^p(X).
\end{equation}
 Then the above implies that $\omega\in  L^2_{\mu-2}(X)$ satisfies $\Delta_{X} \xi = \omega$ for some $\xi\in W_{\mu}^{2,2}(X)$ if and only if for any $\eta\in \mathcal{H}_{-\mu}^p(X)$,
\begin{align}
 \int_X (\omega, \eta) \dvol_g = 0.
\end{align}
When $k\in\dZ_+$, the proof follows from the standard elliptic regularity theory.

The proof of (3) is similar to the proof of \cite[Theorem~4.4]{CCII}, so is omitted.
To prove (4), for any $\omega\in\mathcal{H}_{\mu}^p(X)$, by the asymptotics 
\begin{equation}
|\Phi^*g - g^\fM| = O(\fs^{-\delta}),
\end{equation} we find that $\Delta_\fM \omega \in W^{k,2}_{\mu-\delta-2}(X)$ for all $k\in\dN_0$.
By (3), there is a solution $\eta\in W^{k+2,2}_{\mu-\delta}(X)$ such that 
$\Delta_\fM \eta = \Delta_\fM \omega$
when $r \ge 2 R$. Then $\Delta_\fM (\eta-\omega)=0$ when $r \ge 2 R$. 
Let us pull back $\eta-\omega$ to $\widehat{\fM}$ and write it as $\zeta' + \zeta''$, where each coefficient function of $\zeta'$ has zero $\dT^2$-average, and $\zeta''$ is $\dT^2$-invariant and hence its coefficient functions depend only on $(r,\theta_1)$. By Corollary \ref{Exponential-decay}, $\zeta'$ decays exponentially. 

The next step is to analyze the $\dT^2$-invariant form $\zeta''$. Using separation by variables in the coordinates $(r,\theta_1)$ from Section~\ref{s:model-metric}, we observe that the Fourier expansion of $\zeta''$ can be written as 
\begin{align}
\zeta''=\sum_{j=-\infty}^{\infty} \sum_{i=1}^{\dim \mathcal{Z}_j^p(\widehat{\fM})} c_{j,i} \cdot \omega_{j,i},\label{e:irreducible-coefficients}
\end{align}
where  $\{\omega_{j,i}\}_{1\leq i\leq \dim \mathcal{Z}_j^p(\widehat{\fM})}$ is a basis of $\mathcal{Z}_j^p(\widehat{\fM})$.

First, since $\zeta'' \in L^2_{\mu}(\widehat{\fM})$, the $e^{\i \cdot j \cdot \t1}$ and $e^{-\i \cdot j \cdot \t1}$ components of $\zeta''$ are also in $L^2_{\mu}(\widehat{\fM})$. This implies that $c_{j,i}=0$ for all $j > \mu$. Next, we define the integer $\fq \equiv \ulcorner (\mu -\delta) \urcorner$ and the form
\begin{equation}
\hat{\zeta} \equiv \sum_{j=-\infty}^{\fq-1} \sum_{i=1}^{\dim \mathcal{Z}_j^p(\widehat{\fM})} c_{j,i} \cdot u_{j,i} \cdot e^{\i \cdot m_{j,i}\cdot \theta_1}.
\end{equation}
If $\fq < \mu$, then 
we write $\zeta'' = \zeta_{\fq} + \hat{\zeta}$, where 
\begin{equation}
\zeta_{\fq} \equiv \sum_{i=1}^{\dim \mathcal{Z}_{\fq}^p(\widehat{\fM})} c_{\fq,i} \cdot u_{\fq,i} \cdot e^{\i \cdot m_{\fq,i}\cdot \theta_1} \in \mathcal{Z}_{\fq}^p(\widehat{\fM}).
\end{equation}
If $\fq > \mu$, then $\hat{\zeta}=\zeta''$. 

In the following, we will prove $\hat{\zeta}\in W_{\mu-\delta}^{k,2}(\widehat{\fM})$. A technical issue is that, for example, when $p=0$, 
\begin{equation}
\int_{r=t} (r^k e^{\i \cdot k \cdot \t1}, r^{-k} e^{\i \cdot k \cdot \t1}) d\t1 \wedge d\t2 \wedge \Theta \not= 0.
\end{equation}
To solve this issue, we define
\begin{equation}
\hat{\zeta}_1 \equiv \sum_{j=-\infty}^{-|\fq|-10} \sum_{i=1}^{\dim \mathcal{Z}_j^p(\widehat{\fM})} c_{j,i} \cdot u_{j,i}  \cdot e^{\i \cdot m_{j,i}\cdot \theta_1}.
\end{equation}
Then $\hat{\zeta} = \hat{\zeta}_1 + \hat{\zeta}_2$, where
\begin{equation}
\hat{\zeta}_2 = \sum_{j=-|\fq|-9}^{\fq-1} \sum_{i=1}^{\dim \mathcal{Z}_j^p(\widehat{\fM})} c_{j,i} \cdot u_{j,i}  \cdot e^{\i \cdot m_{j,i}\cdot \theta_1} \in W_{\mu-\delta}^{k,2}(\widehat{\fM}).
\end{equation}
The main improvement is the components of $\hat{\zeta}_1$ are orthogonal to each other. Therefore,
\begin{equation}
\int_{r=t} (\hat{\zeta}_1, \hat{\zeta}_1) d\t1 \wedge d\t2 \wedge \Theta = 8 \pi^3 \cdot \sum_{j=-\infty}^{-|\fq|-10} \sum_{i=1}^{\dim \mathcal{Z}_j^p(\widehat{\fM})} |c_{j,i}|^2 \cdot |u_{j,i}|^2.
\end{equation}

We only need to show that $\hat{\zeta}_1\in W_{\mu-\delta}^{k,2}(\widehat{\fM})$ using the information that $\hat{\zeta}_1\in W_{\mu}^{k,2}(\widehat{\fM})$. To this end, let us fix a large number $r_0\gg 1$.
By Definition \ref{d:Z}, there is a constant $C=C(r_0,\kappa_0, \nu, \mu)>0$ and constants $-2 \le \mathfrak{b}_{j,i} \le 2$ such that for any $j\leq -|\fq| - 10$ and $t > r_0$, we have
\begin{align}
C^{-1} \cdot V^{\mathfrak{b}_{j,i}}(t) \cdot t^j \le |u_{j,i}|(t) \le C \cdot V^{\mathfrak{b}_{j,i}}(t) \cdot t^j.
\end{align}
Therefore, for any $j\leq -|\fq| - 10$, $t_1 > 4 r_0$, and $t_2\in(r_0,2r_0)$,
\begin{align}
\frac{|u_{j,i}|(t_1)}{|u_{j,i}|(t_2)} \leq C \cdot \Big(\frac{V(t_1)}{V(t_2)}\Big)^{\mathfrak{b}_{j,i}} \cdot  \Big(\frac{t_1}{t_2}\Big)^j \le  C \cdot \Big(\frac{V(t_1)}{V(t_2)}\Big)^{2} \cdot  \Big(\frac{t_1}{t_2}\Big)^{-|\fq|-10} \le C \cdot V^{2}(t_1) \cdot  t_1^{-|\fq|-10}.
\end{align}
Note that the constants $C=C(r_0,\kappa_0, \nu, \mu)>0$ are allowed to change line by line. So
\begin{align}
\begin{split}
&\ \int_{r=t_1} (\hat{\zeta}_1, \hat{\zeta}_1) d\t1 \wedge d\t2 \wedge \Theta \\
  \leq & \ C \cdot V^4(t_1) \cdot  t_1^{-2|\fq|-20}	\int_{r_0}^{2r_0} \Big(\int_{r=t_2}(\hat{\zeta}_1, \hat{\zeta}_1) d\t1 \wedge d\t2 \wedge \Theta \Big) d t_2 \\
  \leq &\ C \cdot V^4(t_1) \cdot  t_1^{-2|\fq|-20} \|\hat{\zeta}_1\|^2_{L^2_{\mu}(\widehat{\fM})}.
\end{split}
\end{align}
Then  $\hat{\zeta}_1\in L^2_{\mu-\delta}(\widehat{\fM})$, and the higher order estimate follows from the standard elliptic regularity. 
 \end{proof}

\section{Applications}
In this section, we will apply the above results to prove Theorem~\ref{t:ehf}, Corollary~\ref{t:non-existence},
Theorem~\ref{t:h1}, Corollary~\ref{t:h10}, and Theorem~\ref{t:sharp-order-intro} from the introduction.

\subsection{Existence of harmonic functions}

\begin{proof}[Proof of Theorem \ref{t:ehf}]
From Subsection \ref{ss:a1I}, we see that the function $r^k e^{\i k \theta_1}$ is a harmonic function on the model space $\widehat{\fM}$. Moreover, it descends to $\fM$ by our assumption. Let $\psi$ be a cut-off function on $X$ such that
\begin{align}
\psi
=
\begin{cases}
1, & s \ge 2R_0,
\\
0, & s \le R_0.
\end{cases}	
\end{align}
Since 
\begin{equation}
|\nabla_{g^{\widehat{\fM}}}^l(\Phi^*g- g^{\widehat{\fM}})| = O(\fs^{-l-\mathfrak{n}}) \quad \text{as}\ \fs \to \infty 
\end{equation}
for all $l\in\dN_0$, and $\fs(x)\equiv r(x)\cdot V^{\frac{1}{2}}(x)$,
we have
\begin{equation}
\Delta_g (\psi \cdot r^k e^{\i k \theta_1}) \in W_{\mu-2}^{l,2}(X)
\end{equation}
for any $l\in\dN_0$ and $\mu > k - \mathfrak{n}$. We can choose $\mu$ such that in addition, $\max\{0, k - \mathfrak{n}\}< \mu < k-\epsilon$ and $\mu$ is not an integer. Using the maximum principle, we have $\mathcal{H}_{-\mu}(X)=\{0\}$.
Applying item (2) of Proposition \ref{Weighted-analysis-X}, there exists some $u\in W_{\mu}^{l+2,2}(X)$ which solves the equation
\begin{align}
\Delta_g u= - \Delta_g(\psi \cdot r^k e^{\i k \theta_1}).
\end{align}
Then
\begin{equation}
h_k \equiv \psi \cdot r^k e^{\i k \theta_1} + u
\end{equation}
is harmonic with respect to $g$. Proposition \ref{p:Sobolevemb} then implies \eqref{e:hkexp}.
\end{proof}

\begin{proof}[Proof of Corollary \ref{t:non-existence}]
Recall that
\begin{equation}
|d (r e^{\i\theta_1})|_{g^{\fM}} = \sqrt{2} \cdot V^{-1/2} \to 0,
\end{equation}
as $\fs \to \infty$. This  implies that $|d h_1|_g = o(1)$ as $s\to \infty$, where $s(x)\equiv d_g(x,x_0)$ and $x_0$ is a fixed point in some compact subset $X_R \subset X$. Since $\Ric_g \geq 0$ and $d h_1$ is harmonic, by Bochner's formula we have
\begin{equation}
\Delta_g (|d h_1|_g^2) \le 0.
\end{equation}
Since $d h_1$ is not identically zero, the maximum of $|d h_1|_g^2$ would be achieved in the interior. The strong maximum principle then leads to a contradiction.
\end{proof}

\subsection{Hodge theory on ALG$^*$ manifolds}
\label{s:hodge}
\begin{proof}[Proof of Theorem \ref{t:h1}]

As in \cite[Lemma~6.11]{melrose}, we define a smooth function $f$ such that $f(r)=r$ when $r\le R$ and $f(r)=2R$ when $r\ge 2R$.
Therefore, the map $F : X \rightarrow X$ defined by 
\begin{align}
F(x) = 
\begin{cases}
\Phi( (f(r),\t1,\t2,\t3))   & \mbox{ if } x = \Phi(r, \t1,\t2,\t3)\\
x  & \mbox{ if } x \in X_R\\
\end{cases}
\end{align}
is well-defined and homotopic to the identity. Then any smooth closed 1-form $\eta_1$ on $X$ is cohomologous to $\eta_2\equiv F^*\eta_1$. Clearly, 
\begin{align}
\Phi^* F^* \eta_1 = a_1 d \theta_1 + a_2 d \theta_2 + a_3 d \theta_3,
\end{align}
on the set $r \ge 2 R$, where $a_1, a_2, a_3$ are constants. 
So then $\eta_2 \in  W^{k,2}_{\mu}(X)$ for all $k\in\dN_0$ and $0 < \mu \ll 1$. The mapping $v \mapsto (\eta_2, v)_{L^2(X)}$ is a bounded linear functional on  $\mathcal{H}^1_{-2-\mu}(X)$. By the Riesz representation theorem, there exists  $\eta_3\in \mathcal{H}^1_{-2-\mu}(X)$ such that $(\eta_2, v)_{L^2(X)} = (\eta_3, v)_{L^2(X)}$ for all $v \in \mathcal{H}^1_{-2-\mu}(X)$. By Proposition \ref{Weighted-analysis-X}, there exists an 1-form $\eta_4 \in W^{k+2,2}_{\mu+2}(X)$ such that
\begin{equation}
\eta_2 - \eta_3 = d \delta \eta_4 + \delta d \eta_4.
\end{equation}
The boundary term in the integral
\begin{equation}
\int_{r \le R} \Big((\eta_3,d \delta \eta_3 + \delta d \eta_3) - (d \eta_3, d \eta_3) - (\delta \eta_3, \delta \eta_3)\Big) \dvol_g
\end{equation}
goes to 0 when $R\to \infty$. So we see that $d \eta_3 = \delta \eta_3 =0$.
So $\eta_1$ is also cohomologus to the closed and co-closed 1-form $\eta_5\equiv\delta d \eta_4 \in W^{k,2}_{\mu}(X)$. 

So then $\eta_5$ is a harmonic $1$-form in $W^{k,2}_{\mu}(X)$ for any $\mu$ satisfying $0 < \mu \ll 1$. By Proposition \ref{Weighted-analysis-X}, $\eta_5$ admits a harmonic expansion. By Definition \ref{d:Z},  the leading term $\eta_6$ of the pull-back of $\eta_5$ to $\widehat{\fM}$ is given by 
\begin{align}
\label{eta6}
\eta_6 = (A_0 + B_0 V^2) dz + (A_0' + B_0' V^2) d \bar{z}
+ (C_0 + D_0 V^2) d \t2 + (C_0' V^{-1} + D_0' V) \Theta,
\end{align}
where $A_0, A_0', B_0, B_0', C_0, C_0', D_0, D_0'$ are constants and $z \equiv r e^{\i \t1}$. 
Since $\eta_5$ is both closed and co-closed, we have 
\begin{equation}
\label{e:almost-closed-coclosed}
|d \eta_6|_{g^{\widehat{\fM}}} + |\delta_{\widehat{\fM}} \eta_6|_{g^{\widehat{\fM}}} 
= O (\fs^{-1-\epsilon})
\end{equation}
for a constant $\epsilon>0$, as $\fs \to \infty$. Then \eqref{eta6} and \eqref{e:almost-closed-coclosed} imply that 
\begin{align}
\label{eta60}
\eta_6 = A_0 dz + A_0' d \bar{z} + C_0 d \t2.
\end{align}
By Theorem \ref{t:ehf}, there exists a harmonic function $h_1 : X \rightarrow \CC$ such that $h_1 = z + O(s^{-\epsilon})$  for some $\epsilon > 0$ as $s \to \infty$. 
Then the closed and co-closed 1-form
\begin{align}
\label{eta7}
\eta_7 = \eta_5 - A_0 dh_1 - A_0' d \bar{h}_1
\end{align}
is cohomologous to $\eta_1$. This shows that the natural mapping
\begin{equation}
\{ \omega \in \Omega^1(X) \ | \ d\omega =0, \delta \omega= 0, \Phi^*\omega = \omega_0 + O(\fs^{-\epsilon}) \mbox{ as } \fs \to \infty, \ \omega_0 \in \mathcal{W}^1\} \to H^1_{dR}(X)
\end{equation}
is surjective. To show the injectivity, assume that 
\begin{align}
\omega \in \{ \omega \in \Omega^1(X) \ | \ d\omega =0, \delta \omega= 0, \Phi^*\omega = \omega_0 + O(\fs^{-\epsilon}) \mbox{ as } \fs \to \infty, \ \omega_0 \in \mathcal{W}^1\}
\end{align}
satisfies $\omega = d u$ for a function $u : X \rightarrow \RR$. On the compact subset $\{r \le R \} \subset X$, $u$ is bounded. For each point $\Phi(r, \t1, \t2, \t3)$ on $\{r \ge R\}\subset X$, the path 
\begin{equation}
\gamma_{r, \t1, \t2, \t3} (t) \equiv \Phi(t, \t1, \t2, \t3), t \in [R,r],
\end{equation}
connects $\Phi(R, \t1, \t2, \t3)$ and $\Phi(r, \t1, \t2, \t3)$.
Using
\begin{equation}
u(\Phi(r, \t1, \t2, \t3)) = u(\Phi(R, \t1, \t2, \t3)) + \int_{\gamma_{r, \t1, \t2, \t3}} \omega,
\end{equation}
we have that 
$u = O(s^{1-\epsilon})$ as $s \to \infty$ because the integral of $(\Phi^{-1})^* d \t2$ on $\gamma_{r, \t1, \t2, \t3}$ is $0$. Since $\omega$ is co-closed, from Proposition  \ref{Weighted-analysis-X} 
and Definition \ref{d:Z}, $u$ admits a harmonic expansion 
\begin{align}
u = A_0 + B_0 V + O(s^{- \epsilon})
\end{align}
as $s \to \infty$. Integrating by parts, 
\begin{equation}
0 = \int_{r < R}\Big( (\delta d u, 1) - (d u, d 1) \Big) \dvol_g 
= \int_{r=R} \partial_n u  \ d\sigma_g,
\end{equation}
where $\partial_n$ is the derivative with respect to the unit normal, 
and $d \sigma_g$ is the induced area element. Clearly, 
\begin{align}
\Phi^* \partial_n u &= V^{-\frac{1}{2}} \partial_r (\Phi^* u) + O(\fs^{-1-\epsilon}),\\
\Phi^* d  \sigma_g &= r V^{\frac{1}{2}} d \t1 \wedge d \t2 \wedge \Theta +  O(\fs^{-\epsilon}),
\end{align} 
as $\fs \to \infty$.
This implies that $B_0 = 0$ by taking $R\to \infty$. So $u=A_0$ by the maximum principle. This implies that $\omega=0$, which completes the proof of the injectivity.
\end{proof}

\begin{proof}[Proof of Corollary \ref{t:h10}]
Let
\begin{align}
\omega \in \{ \omega \in \Omega^1(X) \ | \ d\omega =0, \delta \omega= 0, \Phi^*\omega = \omega_0 + O(\fs^{-\epsilon}) \mbox{ as } \fs \to \infty, \ \omega_0 \in \mathcal{W}^1\}.
\end{align}
Recall that
\begin{equation}
|d\theta_2|_{g^{\fM}} = V^{-1/2} = o(1), 
\end{equation}
as $\fs \to \infty$.  If $X$ has non-negative Ricci curvature, then the same argument as in the proof of Corollary \ref{t:non-existence} above would imply that $\omega \equiv 0$. Therefore $b^1(X) = 0$ by Theorem \ref{t:h1}.
\end{proof}

\subsection{Asymptotics of ALG$^*$ gravitational instantons}
\label{s:aALGstar}
We begin with a few remarks about hyperk\"ahler structures. Let $(M, g, I,J,K)$ be a hyperk\"ahler $4$-manifold. Recall that we denote the triple of $2$-forms as $\bm{\omega} = (\omega_1, \omega_2, \omega_3)$, where $\omega_1, \omega_2, \omega_3$ are the K\"ahler forms associated to $I, J, K$, respectively. 
These $2$-forms satisfy
\begin{align}
\label{e:HKtriple}
\omega_1 \wedge  \omega_2 = 0, \ \omega_1 \wedge  \omega_3 = 0,\
 \omega_2 \wedge  \omega_3 = 0, \ \omega_1 \wedge \omega_1 - \omega_2 \wedge \omega_2 =0,\ \omega_1 \wedge \omega_1 - \omega_3 \wedge \omega_3 = 0.
\end{align}
Conversely, any triple of symplectic forms $\omega_i$ satisfying \eqref{e:HKtriple} determines a hyperk\"ahler structure if we replace $\omega_3$ by $-\omega_3$ if necessary. To see this, using the algebraic isomorphism of homogeneous spaces
\begin{align}
\label{algiso}
\SO_0(3,3)/ (\SO(3) \times \SO(3)) \cong
\SL(4,\RR)/ \SO(4) 
\end{align}
(see for example \cite[Chapter~7]{Salamon}), a triple $\bm{\omega}$ uniquely determines a Riemannian metric $g_{\bm{\omega}}$ such that each $\omega_j$ is self-dual with respect to $g_{\bm{\omega}}$ and $\dvol_{g_{\bm{\omega}}} = \frac{1}{2} \omega_1 \wedge \omega_1$.

Next, define
$\mathcal{Q} :  \Lambda^2 \oplus \Lambda^2 \oplus \Lambda^2 
\rightarrow \Lambda^4 \oplus  \Lambda^4\oplus  \Lambda^4\oplus  \Lambda^4\oplus  \Lambda^4
$ by 
\begin{align}
\label{Q}
\mathcal{Q} ( \omega_1, \omega_2, \omega_3) 
= ( \omega_1 \wedge  \omega_2, \ \omega_1 \wedge  \omega_3, \omega_2 \wedge  \omega_3, \omega_1 \wedge \omega_1 - \omega_2 \wedge \omega_2, \omega_1 \wedge \omega_1 - \omega_3 \wedge \omega_3 ).
\end{align}
The kernel of the linearized operator of $\mathcal{Q}$ at a hyperk\"ahler triple will be $\Lambda^2_- \oplus \Lambda^2_- \oplus \Lambda^2_- \oplus E_M$, where $E_M$
is the rank $4$ bundle given by 
\begin{align}
\begin{split}
E_M \equiv \{ \bm{\theta} \in \Lambda^2_+ \oplus \Lambda^2_+ \oplus \Lambda^2_+ \ | \ &\omega_i \wedge \theta_j + \omega_j \wedge \theta_i = 0, i < j \mbox{ and } \omega_1 \wedge \theta_1 = \omega_2 \wedge \theta_2 = \omega_3 \wedge \theta_3\}.
\end{split}
\end{align}
Recall that there is a Dirac-type operator $\mathcal{D} : \Gamma(TM) \rightarrow \Gamma(E_M)$, defined as follows: for every $Y\in\Gamma(TM)$, $\mathcal{D}(Y)$ is the projection of the Lie derivative $\mathcal{L}_Y(\bm{\omega})$ to $E_M$. The operator $- \mathcal{D} \mathcal{D}^*$ can be identified with the Laplacian on functions, since $E_M$ admits a basis of parallel sections; see \cite[Section~3]{CCII}. 

\begin{proof}[Proof of Theorem~\ref{t:sharp-order-intro}]
Using the Fredholm Theory developed above, the argument is very similar to the proof of \cite[Theorem~A]{CCII}. Decompose the difference
\begin{equation}
\label{e:diffw}
 \Phi^* \bm{\omega}^X - \bm{\omega}^{\fM}  =\bm{\eta}_+  + \bm{\eta}_-
\end{equation}
where $\bm{\eta}_+$ is a self-dual triple and $\bm{\eta}_-$ is an anti-self-dual triple with respect to $g^{\fM}$ and the volume form $\frac{1}{2}(\omega_1^{\fM} \wedge \omega_1^{\fM})$.   We choose an irrational number $\epsilon\in(0,\mathfrak{n})$, and view the right hand side of \eqref{e:diffw} as an element of $W^{k,2}_{-\epsilon}(\fM_{2\nu}(R))$ for any $k \in \ZZ_+$.  For sufficiently large $R$, expanding \eqref{Q} yields
\begin{align}
\mathcal{Q}(\Phi^*\bm{\omega}^X) = \mathcal{Q}(\bm{\omega}^{\fM}) 
+ \mathscr{L}_{\bm{\omega}^{\fM}} ( \Phi^* \bm{\omega}^X - \bm{\omega}^{\fM})
+ \mathcal{N}_{\bm{\omega}^{\fM}}(  \Phi^* \bm{\omega}^X - \bm{\omega}^{\fM}),
\end{align}
where $\mathscr{L}_{\bm{\omega}^{\fM}}$ is the linearized operator at $\bm{\omega}^{\fM}$, 
and $\mathcal{N}_{\bm{\omega}^{\fM}}$ are the nonlinear terms. We then have 
\begin{align}
\label{e:Almost-in-E}
\mathscr{L}_{\bm{\omega}^{\fM}} ( \bm{\eta}_+) =  \mathscr{L}_{\bm{\omega}^{\fM}} ( \bm{\eta}_+  + \bm{\eta}_-) 
= -  \mathcal{N}_{\bm{\omega}^{\fM}}( \bm{\eta}_+  + \bm{\eta}_-).
\end{align}
From the structure of \eqref{Q}, the right hand side of \eqref{e:Almost-in-E} must be in $W^{k,2}_{-2\epsilon}(\fM_{2\nu}(R))$. Therefore, the projection $\bm{\theta}$ of $\bm{\eta}_+$ to $E_{\fM_{2\nu}(R)}$ satisfies $\bm{\theta} \in W^{k,2}_{-\epsilon}(\fM_{2\nu}(R))$
and $\bm{\eta}_+ - \bm{\theta} \in  W^{k,2}_{-2\epsilon}(\fM_{2\nu}(R))$.

By Proposition~\ref{Weighted-analysis-X} (3), there exists a bounded linear operator 
\begin{align}
G_{g^{\fM}} : W^{k,2}_{-\epsilon}( \Gamma(E_{\fM_{2\nu}(R)})) \rightarrow  W^{k+2,2}_{-\epsilon+2}( \Gamma(E_{\fM_{2 \nu}(R)}))
\end{align}
such that 
\begin{align}
- \mathcal{D}_{g^{\fM}} \mathcal{D}^*_{g^{\fM}} G_{g^{\fM}} \bm{\theta} = \bm{\theta}.
\end{align}   
Consider the vector field $Y =  \mathcal{D}^*_{g^{\fM}} G_{g^{\fM}} \bm{\theta}$ which satisfies $Y \in W^{k+1,2}_{1 - \epsilon}(\fM_{2\nu}(R))$. By elliptic regularity, $Y$ is smooth, and by Proposition~\ref{p:Sobolevemb}, $|\nabla^m Y| = O(\fs^{1 - m - \epsilon})$ for any $m \in \dN_0$, as $\fs \to \infty$. It is important to point out that $Y$ does not depend on the choice of $k$.
Define $\Phi_t : \fM_{2\nu}(2R) \rightarrow \fM_{2\nu}(R)$ by 
$\Phi_t(x) = \exp_{g^{\fM},x}(tY_x)$. If $0 \leq t \leq 1$, since $1- \epsilon < 1$, the vector $tY_x$ has norm much smaller than the conjugate radius at $x$ (which is comparable to $\fs(x)$), so $\Phi_t$ is a diffeomorphism onto its image for $R$ sufficiently large. Then we have 
\begin{align}
\Phi_1^* \bm{\omega}^{\fM} - \bm{\omega}^{\fM} - \mathcal{L}_Y \bm{\omega}^{\fM}  \in W^{k,2}_{-2 \epsilon}(\fM_{2\nu}(R)).
\end{align}
Let $\Phi' = \Phi \circ \Phi_1$. From \eqref{e:diffw}, we then have 
\begin{align}
\begin{split}
(\Phi')^*\bm{\omega}^{X} -  \bm{\omega}^{\fM} - \bm{\eta}_+  - \bm{\eta}_-  - \mathcal{L}_Y \bm{\omega}^{\fM}  = \Phi_1^* \Phi^* \bm{\omega}^{X} - \Phi^*\bm{\omega}^{X} - \mathcal{L}_Y \bm{\omega}^{\fM}\\
 = ( \Phi_1^* - \Id) (\Phi^* \bm{\omega}^{X} - \bm{\omega}^{\fM})+ \Phi_1^* \bm{\omega}^{\fM} - \bm{\omega}^{\fM} - \mathcal{L}_Y \bm{\omega}^{\fM}
  \in W^{k,2}_{-2 \epsilon}(\fM_{2\nu}(R)).
\end{split}
\end{align}
Similar to \eqref{e:Almost-in-E} above, we have the expansion
\begin{align}
\label{e:Almost-in-E2}
\mathscr{L}_{\bm{\omega}^{\fM}} ( \bm{\eta}_+  + (\mathcal{L}_Y \bm{\omega}^{\fM})^+) = \mathscr{L}_{\bm{\omega}^{\fM}} ( \bm{\eta}_+  + \bm{\eta}_- + \mathcal{L}_Y \bm{\omega}^{\fM}) = -\mathcal{N}_{\bm{\omega}^{\fM}}( \bm{\eta}_+  + \bm{\eta}_- + \mathcal{L}_Y \bm{\omega}^{\fM}).
\end{align}
From the structure of \eqref{Q}, the right hand side of \eqref{e:Almost-in-E2} must be in $W^{k,2}_{-2\epsilon}(\fM_{2\nu}(R))$. Since the projection of $\bm{\eta}_+  + (\mathcal{L}_Y \bm{\omega}^{\fM})^+$ to $E_{\fM_{2\nu}(R)}$ 
is $\bm{\theta} + \mathcal{D}_{g^{\fM}} Y = 0$,  we see that $\bm{\eta}_+  + (\mathcal{L}_Y \bm{\omega}^{\fM})^+ \in W^{k,2}_{-2\epsilon}(\fM_{2\nu}(R))$. In other words, redefining $\bm{\eta}_{\pm}$, we may assume that 
\begin{align}
(\Phi')^* \bm{\omega}^X-\bm{\omega}^{\fM} = \bm{\eta}_+ + \bm{\eta}_-
\end{align}
with $\bm{\eta}_- \in W^{k,2}_{-\epsilon}(\fM_{2\nu}(R))$, 
and $\bm{\eta}_+ \in W^{k,2}_{-2\epsilon}(\fM_{2\nu}(R))$.
Since $\bm{\eta}_-$ is anti-self-dual with respect to $\bm{\omega}^{\fM}$, 
\begin{align}
\Delta_{g^{\fM}} \bm{\eta}_- = (- *_{g^{\fM}} d *_{g^{\fM}} d - d*_{g^{\fM}}d*_{g^{\fM}})  \bm{\eta}_-
=- *_{g^{\fM}} d *_{g^{\fM}} d    \bm{\eta}_- + d*_{g^{\fM}}d   \bm{\eta}_-.
\end{align}
But since both triples are closed, $d \bm{\eta}_- = - d \bm{\eta}_+$,  therefore, $\Delta_{g^{\fM}} \bm{\eta}_- \in W^{k-2,2}_{- 2 \epsilon - 2}(\fM_{2\nu}(R))$. 
By (3) and (4) in Proposition~\ref{Weighted-analysis-X}, if $\epsilon < \frac{1}{2}$, then $\bm{\eta}_-\in W^{k,2}_{-2\epsilon}(\fM_{2\nu}(2R))$. We can replace $(\Phi, \epsilon)$ with $(\Phi', 2\epsilon)$ and repeat this process until $\epsilon > \frac{1}{2}$. If $\frac{1}{2} < \epsilon < 1$, then $\bm{\eta}_-$ is the sum of a triple in $\mathcal{Z}^{2,-}_{-1}\otimes\RR^3$ and a triple in $W^{k,2}_{-2\epsilon}(\fM_{2\nu}(2R))$. By Lemma \ref{l:Z2-}, all the non-zero forms in $\mathcal{Z}^{2,-}_{-1}$ are not invariant under the $\dZ_2$-action $\iota:\widehat{\fM}\to\widehat{\fM}$. So we see that we can replace $(\Phi, \epsilon)$ with $(\Phi', 2\epsilon)$ and repeat this process again. If $1 < \epsilon < 2$, then $\bm{\eta}_-$ is the sum of $\bm{\eta}_{\mathcal{Z}} \in \mathcal{Z}^{2,-}_{-2}\otimes\RR^3$ and a triple in $W^{k,2}_{-\epsilon-1}(\fM_{2\nu}(2R))$. Moreover, the decay rate of $d\bm{\eta}_{\mathcal{Z}}$ is strictly larger than $3$. So by the explicit definition of $\mathcal{Z}^{2,-}_{-2}$ in Lemma \ref{l:Z2-}, $|\nabla^k \bm{\eta}_{\mathcal{Z}}| = O(\fs^{-2-k})$, as $\fs \to \infty$, for any $k\in\dN_0$. The above is a finite iteration, so $\bm{\eta}_- - \bm{\eta}_{\mathcal{Z}} \in W^{k,2}_{-\epsilon-1}(\fM_{2\nu}(2R))$ for any $k \in \dN_0$. Then by Proposition~\ref{p:Sobolevemb}, $|\nabla^k (\bm{\eta}_- - \bm{\eta}_{\mathcal{Z}})| = O(\fs^{-\epsilon-1-k})$, as $\fs \to \infty$, for any $k\in\dN_0$. Therefore, by replacing $\Phi$ with $\Phi'$,  \eqref{algstaromega} is satisfied for $\mathfrak{n} = 2$.

\end{proof}

\appendix

\section{Computations on the ALG$^*$ model space}
In this section, we derive some expansions of harmonic forms on the model space $(\widehat{\fM}_{\nu}(R), g^{\widehat{\fM}}_{\kappa_0,L})$. By scaling, we may assume that $L = 1$.
For simplicity, we will denote $\widehat{\fM} \equiv \widehat{\fM}_{\nu}(R)$. 
Define an orthonormal basis of $1$-forms by
\begin{align}
\{ e^0, e^1, e^2, e^3 \} = \{ V^{1/2} dr, V^{1/2}r d\t1, V^{1/2} d\t2, V^{-1/2} \Theta\},
\end{align}
and use the orientation $e^0 \wedge e^1 \wedge e^2 \wedge e^3$. The following formulas are straightforward to verify, and will be used throughout the appendix. We have
\begin{align}
\begin{split}
d e^0 &= 0, \quad d e^1 = \p_r(V^{1/2}r) V^{-1} r^{-1} e^0 \wedge e^1,\\
d e^2 & = \p_r (V^{1/2}) V^{-1} e^0 \wedge e^2, \quad d e^3 = \p_r ( V^{-1/2}) e^0 \wedge e^3 + \frac{\nu}{2\pi}V^{-3/2} r^{-1} e^1 \wedge e^2,
\end{split}
\end{align}
and
\begin{align}
d ( e^0 \wedge e^1) &= d (e^0 \wedge e^2 ) = d (e^2 \wedge e^3) =0, \\
d (e^1 \wedge e^3) &= V^{-1/2} r^{-1} e^0 \wedge e^1 \wedge e^3,\\
d(e^1 \wedge e^2) & = \Big( V^{-3/2} \p_r(V) + V^{-1/2} r^{-1} \Big) e^0 \wedge e^1 \wedge e^2,\\
d(e^0 \wedge e^3) & = - \frac{\nu}{2\pi}V^{-3/2} r^{-1}  e^0 \wedge e^1 \wedge e^2.
\end{align}
\subsection{Laplacian on functions}

First, we need to characterize the harmonic functions of the form
\begin{align}
\label{functans}
h = f(r) e^{\i  k \t1} , \quad k \in\dZ,
\end{align}
with prescribed growth order.
The following lemma identifies the space $\mathcal{Z}_{\fq}^0(\widehat{\fM})$.

 \begin{lemma}
\label{l:Z0}
Given an integer $\fq\in\dZ$, let $\mathcal{Z}_{\fq}^0(\widehat{\fM})$ be the linear space with a basis
 \begin{align}
 \begin{split}
& \cB_0^0 \equiv \{1,\ V\},
 \\
&  \cB_{\fq}^0 \equiv \{r^{\fq}e^{\i \cdot \fq \cdot \t1},\ r^{\fq}e^{-\i \cdot \fq \cdot \t1}\}, \quad \fq \in\dZ\setminus\{0\}.
\end{split}
\end{align}	
If a function $h$ satisfies \eqref{functans} for some $k\in\dZ$ and solves $\Delta h = 0$ on $\widehat{\fM}$, then \begin{align}
h\in
\begin{cases}
\mathcal{Z}_{0}^0(\widehat{\fM}), & \text{if}\   k = 0,
\\
\mathcal{Z}_{k}^0(\widehat{\fM}) \oplus \mathcal{Z}_{-k}^0(\widehat{\fM}), & \text{if} \ k \in\dZ\setminus\{0\}.
\end{cases}	
\end{align}
\end{lemma}

\begin{proof}Taking the differential $dh$ of $h$,
we have that
\begin{align}
d h = e^{\i  k \t1} ( f' dr + \i k f d \t1) =  e^{\i  k \t1} ( f' V^{-1/2} e^0 + \i k f V^{-1/2} r^{-1} e^1).
\end{align}
Then
\begin{align}
\begin{split}
* dh & =   e^{\i k \t1} ( f' V^{-1/2} e^1 \wedge e^2 \wedge e^3  - \i  k f V^{-1/2} r^{-1}  e^0 \wedge e^2 \wedge e^3  )\\
& =  e^{\i k \t1} ( f' r  d \t1 \wedge d \t2 \wedge \Theta  - \i k f r^{-1}   dr \wedge d \t2 \wedge \Theta  ).
\end{split}
\end{align}
Applying $d$,
\begin{align}
\notag
d * dh &=   e^{\i k \t1} \i k d \t1 \wedge (  - \i k f r^{-1}   dr \wedge d \t2 \wedge \Theta  )
+   e^{\i k \t1}  \p_r(f' r) dr  \wedge  d \t1 \wedge d \t2 \wedge \Theta \\
& =  e^{\i k \t1} \Big(   \p_r(f' r) -  k^2 f r^{-1} \Big)  V^{-1} r^{-1}  e^0  \wedge  e^1 \wedge e^2 \wedge e^3.
\end{align}
Finally, we have
\begin{align}
\Delta h = - * d * d h &= -  e^{\i k \t1} \Big(  f'' + r^{-1} f'  -  k^2 r^{-2}f   \Big) V^{-1}.
\label{Expansion-function}
\end{align}
The ODE for a harmonic function is therefore
\begin{align}
 f'' + r^{-1} f'  -  k^2 r^{-2}f   = 0,
\end{align}
and the solutions are given by 
\begin{align}
f(r) =
\begin{cases}
C_1 + C_2 V & k = 0\\
C_1 r^k + C_2 r^{-k} & k \neq 0.
\end{cases}
\end{align}

\end{proof}

\subsection{Forms of Type I on the model space}
\label{ss:a1I}
 A $1$-form $\omega$ is said to be of {\it Type I} if $\omega$ satisfies
\begin{align}
\omega = e^{\i k \t1 } \left( f(r) e^0 + a(r) e^1 \right).\label{e:type-I-omega-a-f}
\end{align}
This subsection studies the Type I solutions of $\Delta \omega = 0$.

Let us denote $dz = dr + \i r d\theta_1$. 
The following lemma characterizes the Type I solutions of $\Delta \omega = 0$.

\begin{lemma}
\label{l:Z1I}
Given an integer $\fq\in\dZ$, let $\mathcal{Z}_{\fq}^{1,\I}(\widehat{\fM})$ be the linear space with a basis 
 \begin{align}
 \begin{split}
 \cB_0^{1,\I} &\equiv  \left\{ dz,\ V^2 dz,\ d\bar{z},\ V^2 d\bar{z}\right\},  \\
  \cB_{\fq}^{1,\I} &\equiv  \Big\{e^{ \i (\fq+1)  \t1 } r^{\fq}\cdot dz,\   e^{ - \i (\fq+1)  \t1 }(\nu - 4 \fq \pi V )  r^{\fq} \cdot dz,\\
\ & \ \ \ \ \ \ e^{-\i (\fq+1)  \t1 } r^{\fq}\cdot d\bar{z},\ e^{\i (\fq+1)  \t1 }(\nu  - 4 \fq \pi V  ) r^{\fq} \cdot d\bar{z}
\Big\}, \quad \fq \neq 0.
\end{split}
\end{align}
If a complex-valued $1$-form $\omega$ of type I satisfies \eqref{e:type-I-omega-a-f} for some $k\in\dZ$ and solves $\Delta \omega = 0$ on $\widehat{\fM}$, then
\begin{align}
\omega\in
\begin{cases}
	\mathcal{Z}_{-1}^{1,\I}(\widehat{\fM}), & \text{if}\  k = 1, 
	\\
	\mathcal{Z}_{k-2}^{1,\I}(\widehat{\fM})\oplus \mathcal{Z}_{-k}^{1,\I}(\widehat{\fM}), & \text{if} \ k \in \dZ\setminus\{1\}.
\end{cases}
\end{align}
\end{lemma}

\begin{proof}
 In the proof, in stead of directly analyzing Type I, we will first reduce \eqref{e:type-I-omega-a-f} by analyzing the following ansatz,
\begin{align}
\omega = e^{\i k \t1 } \left( f(r) e^0 + a(r) e^1 \right) = e^{\i k \t1 } u(r) \left( e^{-\i \t1} V^{1/2} dz\right). \label{e:type-I-u-ansatz} \end{align} Then
\begin{align}\omega = e^{\i k \t1} u(r) ( e^0 + \i e^1),
 \end{align}
 that is, $a(r) = \i f(r)$.
  Note that for the complex structure $I$ defined as above,
we have $I^*(e^0) = - e^1$. Then
\begin{align}
e^0 - \i I^* e^0 = e^0 + \i e^1 \in \Lambda^{1,0}_I.
\end{align}
Since the metric is K\"ahler, the Hodge Laplacian preserves the type of forms,
so \eqref{e:type-I-u-ansatz} is a natural ansatz.
 
To begin with, let us compute the Laplacian of $\omega$. First,
\begin{align}
* \omega = e^{\i k \t1 } u(r) \Big(   e^1 \wedge e^2 \wedge e^3 - \i  e^0 \wedge e^2 \wedge e^3 \Big).
\end{align}
This is
\begin{align}
* \omega = e^{\i k \t1 } u(r) V^{1/2} \Big(   r  d \t1 \wedge d\t2 \wedge \Theta - \i  dr \wedge d \t2 \wedge \Theta \Big).
\end{align}
Then
\begin{align}
\notag
d * \omega &= e^{\i k \t1 }\i k u V^{1/2} d \t1 \wedge  \Big( - \i     dr \wedge d \t2 \wedge \Theta  \Big)
+ e^{\i k \t1 } \p_r( r V^{1/2}u )   dr \wedge  d \t1 \wedge d\t2 \wedge \Theta\\
& =  e^{\i k \t1} \Big(     \p_r( r V^{1/2}u ) - k u V^{1/2} \Big) V^{-1}r^{-1}  e^0 \wedge    e^1 \wedge  e^2 \wedge e^3,
\end{align}
which implies that
\begin{align}
\delta \omega = - * d * \omega =  -  e^{\i k \t1} \Big(   \p_r( r u V^{1/2}) - k u V^{1/2} \Big) V^{-1}r^{-1}.
\end{align}
We define
\begin{align}
p(r) =  \Big(  \p_r( r V^{1/2}u ) - k u V^{1/2} \Big) V^{-1}r^{-1}.
\end{align}
Next, we have
\begin{align}
d \delta \omega &=  - \Big( e^{\i k \t1 } \i k p d \t1 +  e^{\i k \t1 } p' dr \Big),
\end{align}
or
\begin{align}
\label{ddelta1}
d \delta \omega &= - e^{\i k \t1} \Big( p' V^{-1/2} e^0 + \i k  p  V^{-1/2} r^{-1} e^1 \Big).
\end{align}
Next, we compute
\begin{align}
\begin{split}
d \omega &= e^{\i k \t1 }  \Big\{ \i k d \t1 \wedge \Big( u V^{1/2} (dr + \i r d\t1) \Big)
+ \p_r( r V^{1/2} u) dr \wedge ( \i d \t1) \Big\}\\
& =  \i e^{\i k \t1 } \Big(    \p_r( r V^{1/2}u )   -  k u V^{1/2} \Big) V^{-1} r^{-1} e^0 \wedge e^1 \equiv  \i  e^{\i k \t1 } p  e^0 \wedge e^1.
\end{split}
\end{align}
Noting that $* e^0 \wedge e^1 = e^2 \wedge e^3$, applying Hodge star, we have
\begin{align}
* d \omega &=  \i e^{\i k \t1 } p   e^2 \wedge e^3,
\end{align}
which implies that 
\begin{align}
\begin{split}
d * d \omega &=  \i e^{\i k \t1 } \i  k p d \t1 \wedge     e^2 \wedge e^3  +  \i e^{\i k \t1 } p'  dr  \wedge e^2 \wedge e^3\\
& = e^{\i k \t1} \Big\{   - kp V^{-1/2} r^{-1} e^1 \wedge     e^2 \wedge e^3
+ \i p' V^{-1/2} e^0  \wedge e^2 \wedge e^3 \Big\}.
\end{split}
\end{align}
Applying Hodge star,
\begin{align}
* d * d \omega &=  e^{\i k \t1 }  \Big(
   k p V^{-1/2} r^{-1}  e^0 + \i p' V^{-1/2} e^1 \Big).
\end{align}
Combining with \eqref{ddelta1} from above, we see that
\begin{align}
\notag
\Delta \omega &= - e^{\i k \t1}  \Big\{ p' V^{-1/2} e^0 + \i k  p  V^{-1/2} r^{-1} e^1 \Big\}
-   e^{\i k \t1 }  \Big(    k p V^{-1/2} r^{-1}  e^0 + \i p' V^{-1/2} e^1 \Big) \\
& = -  e^{\i k \t1}   \Big(  p' + k r^{-1} p \Big) V^{-1/2}(  e^0 + \i  e^1 ).
\end{align}
Assume that $\omega$ is a harmonic $1$-form. Then we have the homogeneous first order system
\begin{align}
 p' + k r^{-1} p &= 0.
\end{align}
The general solution is
\begin{align}
p & = C_1 r^{-k}.
\end{align}
Recalling that
\begin{align}
p(r) =  \Big(   \p_r( r V^{1/2} u ) - k u V^{1/2} \Big) V^{-1}r^{-1},
\end{align}
we have the ODE
\begin{align}
   \p_r( r u V^{1/2}) - k u V^{1/2} = C_1 V r^{-k+1}.
\end{align}
Let us define $\tilde{u} \equiv V^{1/2} u$, then this ODE can be written as
\begin{align}
\p_r (r \tilde{u}) - k  \tilde{u} = C_1 V r^{-k+1}.\label{e:tilde-u-inhomogeneous}
\end{align}
The general solution of the associated homogeneous equation
\begin{align}
\p_r (r \tilde{u}) - k  \tilde{u} = 0,
\end{align}
is $\tilde{u} = C_2 r^{k-1}$.

If $k = 1$, then the ODE \eqref{e:tilde-u-inhomogeneous} becomes
\begin{align}
 r \tilde{u}'= C_1 V,
\end{align}
which has a particular solution 
\begin{align}
\tilde{u} = C_1\Big( \kappa_0 \log(r) + \frac{ \nu}{4 \pi} (\log(r))^2 \Big).
\end{align}
Therefore, for $k = 1$, we have the general solution
\begin{align}
u(r) = &\ V^{-1/2} \left( C_2 +  C_1\left( \kappa_0 \log(r) + \frac{ \nu}{4 \pi} (\log(r))^2 \right) \right)
\nonumber\\
= &\ V^{-1/2}\left(\left(C_2 -\kappa_0^2\cdot \frac{\pi}{\nu}\cdot C_1\right)  + \frac{\pi}{\nu}\cdot C_1 \cdot V^2 \right).
\end{align}
Letting 
\begin{align}\widetilde{C}_1 \equiv \frac{\pi}{\nu}\cdot C_1 \quad \text{and} \quad 
\widetilde{C}_2 \equiv C_2 -\kappa_0^2\cdot \frac{\pi}{\nu}\cdot C_1,\end{align}
 we can then write the general solution for $k =1$ as
\begin{align}
u(r) = \widetilde{C}_2 V^{-1/2} + \widetilde{C}_1  V^{3/2}.
\end{align}

Next, we assume that $k \neq 1$. Then we need to solve
\begin{align}
\p_r (r \tilde{u}) - k  \tilde{u} = C_1 V r^{-k+1 }.
\end{align}
Using the integrating factor $r^{1-k}$, we see that
the general solution is
\begin{align}
\begin{split}
\tilde{u} &= C_2 r^{k-1} - \frac{C_1 r^{1-k}}{4 \pi (k-1)^2} \Big(
\frac{\nu}{2} + 2 (k-1) \pi \kappa_0 + \nu (k-1) \log(r) \Big) \\
& =  C_2 r^{k-1} - \frac{C_1 r^{1-k}}{4 \pi (k-1)^2} \Big(
\frac{\nu}{2} + 2 (k-1) \pi V \Big).
\end{split}
\end{align}
Letting 
\begin{equation}
\widetilde{C}_1 \equiv \frac{C_1 r^{1-k}}{4 \pi (k-1)^2},
\end{equation}
we therefore have the general solution for $k \neq 1$:
\begin{align}
u =  C_2 V^{-1/2} r^{k-1} + \widetilde{C}_1 r^{1-k} \Big(
\nu V^{-1/2} + 4 (k-1) \pi V^{1/2} \Big).
\end{align}

Notice that the Laplacian is a real operator, so we also have conjugate solutions. This completes the proof.
\end{proof}

\begin{remark}
We note that the solutions with $\widetilde{C}_1 = 0$ are exactly the ones with $p = 0$, so by the above computations, they are exactly the solutions which satisfy $d \omega =0$ and $\delta \omega = 0$. The solutions with $\widetilde{C}_1 \neq 0$ are neither closed nor co-closed.
\end{remark}

\subsection{One-forms of Type II on the model space}
A $1$-form $\omega$ is said to be of {\it Type II} if
\begin{align}
\omega = e^{\i k \t1 } \Big(  b(r) e^2  + c(r)  e^3 \Big). \label{e:type-II-omega-b-c}
\end{align}  
Then we have the following lemma.

\begin{lemma}
\label{l:Z1II}
Given an integer $\fq\in\dZ$, let $\mathcal{Z}_{\fq}^{1,\II}(\widehat{\fM})$ be the linear space with a basis 
 \begin{align}
\notag
 \cB_0^{1,\II} \equiv & \ \Big\{   -  d \t2 + \i V^{-1} \Theta,\  - V^2 d \t2 + \i V  \Theta,   \ - d \t2 - \i V^{-1} \Theta,\  - V^2 d \t2 - \i V \Theta\Big\},
 \\  \notag
  \cB_{\fq}^{1,\II} \equiv & \ \Big\{e^{ \i \fq  \t1 } r^{\fq}\cdot ( - d \t2 + \i V^{-1} \Theta),\ e^{-\i \fq  \t1 }  r^{\fq}\cdot ( -  d \t2 - \i V^{-1} \Theta),\\
&e^{ - \i \fq  \t1 }(\nu V^{-\frac{1}{2}} - 4 \fq \pi V^{\frac{1}{2}} ) ) r^{\fq} \cdot  ( - V^{\frac{1}{2}} d \t2 + \i V^{-\frac{1}{2}} \Theta),\\
\ &  e^{\i \fq  \t1 }(\nu V^{-\frac{1}{2}} - 4 \fq \pi V^{\frac{1}{2}} ) ) r^{\fq} \cdot  ( - V^{\frac{1}{2}} d \t2 - \i V^{-\frac{1}{2}} \Theta)
\Big\}, \quad \fq \in \dZ \setminus\{0\}. \notag
\end{align}	
If a complex-valued $1$-form $\omega$ of type II satisfies \eqref{e:type-II-omega-b-c} for some $k\in\dZ$ and solves $\Delta \omega = 0$ on $\widehat{\fM}$, then 
\begin{align}
\omega \in
\begin{cases}
\mathcal{Z}_0^{1,\II}(\widehat{\fM}), & \text{if}\ k = 0,
\\
\mathcal{Z}_{k}^{1,\II}(\widehat{\fM}) \oplus \mathcal{Z}_{-k}^{1,\II}(\widehat{\fM}), & \text{if}\ k \in \dZ\setminus\{0\}.
\end{cases}	
\end{align}
\end{lemma}

The proof relies on the following relationship between Type II forms and Type I forms:
\begin{lemma}   Any form $\omega_2$ of Type II may be written as $J^*(\omega_1)$, where $\omega_1$ is a form of type I, for the hyperk\"ahler complex structure $J$ and the dual linear operator $J^*$. So the decay rates of harmonic Type II forms are the same as Type I forms.
\end{lemma}
\begin{proof}
We use the complex structure $J$ given by
\begin{align}
J^*( dx ) = - d \t2, \quad J^*(V^{1/2} dy ) = V^{-1/2} \Theta.
\end{align}
Recall that
\begin{align}
e^0 + \i e^1 = V^{1/2} ( dr + \i r d\t1) = V^{-1/2} e^{-\i \t1} ( dx + \i dy).
\end{align}
We then have
\begin{align}
J^* ( e^0 + \i e^1) &= V^{1/2}e^{-\i \t1} ( J dx + \i J dy) \nonumber\\
& =  V^{1/2} e^{-\i \t1} ( - d \t2 + \i V^{-1} \Theta) = - e^{-\i \t1}( e^2 - \i e^3).
\end{align}
So if $\omega_1$ is a $1$-form of type I and satisfies 
\begin{align}
\omega_1 = e^{\i k \t1} u(r) ( e^0 + \i e^1),
\end{align}
then
\begin{align}
J^*(\omega_1) & = - e^{\i (k-1) \t1}  u(r) ( e^2 - \i e^3), \\
J^*(\overline{\omega_1}) & = - e^{\i (1-k) \t1} \overline{u(r)} ( e^2 + \i e^3).
\end{align}
On $1$-forms, $\Delta = \nabla^* \nabla$, and since $J$ is parallel,
$\omega_2$ is harmonic if and only if $\omega_1$ is harmonic.
These generate all Type II harmonic $1$-forms.
\end{proof}

\subsection{$2$-forms on the model space}
Let us define the $2$-forms, 
\begin{align}
\omega^1_{\pm} &=   e^0 \wedge e^1 \pm  e^2 \wedge e^3, \quad \omega^2_{\pm} =   e^0 \wedge e^2 \mp  e^1 \wedge e^3, \quad \omega^3_{\pm} =   e^0 \wedge e^3 \pm  e^1 \wedge e^2.
\end{align}
Then we have
$* \omega^i_{\pm} = \pm \omega^i_{\pm}$, $i = 1, 2, 3$.
Note that
\begin{align}
\omega^1_{\pm} &=   V r dr  \wedge d \t1 \pm  d \t2 \wedge \Theta, \\
\omega^2_{\pm} &=   V dr \wedge d \t2 \mp  r d\t1 \wedge \Theta, \\
\omega^3_{\pm} &=   dr \wedge \Theta \pm  Vr d\t1 \wedge d \t2.
\end{align}
We compute the exterior derivatives
\begin{align}
d \omega^1_{\pm} & = 0,\\
d \omega^2_{\pm} & = \mp dr \wedge d \t1 \wedge \Theta = \mp V^{-1/2} r^{-1} e^0 \wedge e^1 \wedge e^3,\\
d \omega^3_{\pm} & = \Big( -\frac{ \nu}{2\pi} \pm (Vr)'\Big) dr \wedge d \t1 \wedge d \t2
 = \Big( -\frac{ \nu}{2\pi} \pm (Vr)'\Big) V^{-3/2} r^{-1} e^0 \wedge e^1 \wedge e^2.
\end{align}
If we expand further, we have
\begin{align}
d \omega^3_{-} & =- \Big( \frac{\nu}{ \pi} + V\Big) dr \wedge d \t1 \wedge d \t2
 = -\Big( \frac{\nu}{ \pi} + V \Big) V^{-3/2} r^{-1} e^0 \wedge e^1 \wedge e^2.
\end{align}
\begin{remark} Recall that $\omega^i_+$ are not the hyperk\"ahler forms. For this, we should define $E^0 = V^{-1/2} dx,  E^1 = V^{-1/2} dy$. But for the following calculations, we can freely choose an orthonormal basis.
\end{remark}
\begin{remark}  If $\omega$ is self-dual, then we may also write
$\omega = a \omega_I + b \omega_J
+ c \omega_K$, and hence 
\begin{align}
\Delta \omega =  (\Delta a) \omega_I
+  (\Delta b) \omega_J + (\Delta c)  \omega_K.
\label{Expansion-2-form-self-dual}
\end{align}
Since we already know the result about harmonic functions, we can define 
\begin{align}
\label{e:Z2+}
\mathcal{Z}_{\fq}^{2,+}(\widehat{\fM})\equiv \mathcal{Z}_{\fq}^0(\widehat{\fM}) \otimes (\RR \omega_I\oplus\RR \omega_J\oplus\RR \omega_K).
\end{align}
In the following we only need to consider ASD $2$-forms.
\end{remark}
We want $\omega$ to be invariant under the $S^1$-action.
Let us consider the following general ansatz
\begin{align}
\omega = e^{\i k \t1 } \Big( a(r) \w^1_- + b(r) \w^2_-  + c(r)  \w^3_- \Big).\label{e:S1-invarint-2-form-omega}
\end{align}
Then we prove the following lemma.

\begin{lemma}
\label{l:Z2-}
Given an integer $\fq\in\dZ$, let $\mathcal{Z}_{\fq}^{2,-}(\widehat{\fM})$ be the linear space with a basis
\begin{enumerate}
\item For every $\fq\in\dZ\setminus\{0, 1, -1\}$,
\begin{align*}
 \begin{split}
 \cB_{\fq}^{2,-} \equiv & \ \Big\{r^{\fq} e^{\i \fq \theta_1}\omega_-^1, r^{\fq} e^{-\i \fq \theta_1}\omega_-^1,\\
&\frac{\i}{\fq - 1} r^{\fq} \Big( - \frac{ (\fq - 1 ) V}{ 2 \fq} + \frac{ 2 \fq -1 }{2 \fq^2}\frac{\nu}{2\pi} + \frac{(\fq-1)^2}{4 \fq^3} \frac{\nu^2}{4\pi^2} V^{-1} \Big) e^{\i (1-\fq) \t1}\omega^2_-\\
&+ r^{\fq} \Big( - \frac{V}{2 \fq} + \frac{1}{2 \fq^2}\frac{\nu}{2\pi} - \frac{\fq + 1}{4 \fq^3} \frac{\nu^2}{4\pi^2} V^{-1} + \frac{1}{4 \fq^3} \frac{\nu^3}{8\pi^3} V^{-2}\Big) e^{\i (1 - \fq) \t1}\omega^3_-, \\
&-\frac{\i}{\fq - 1} r^{\fq} \Big( - \frac{ (\fq - 1 ) V}{ 2 \fq} + \frac{ 2 \fq -1 }{2 \fq^2}\frac{\nu}{2\pi} + \frac{(\fq-1)^2}{4 \fq^3} \frac{\nu^2}{4\pi^2} V^{-1} \Big) e^{-\i (1-\fq) \t1}\omega^2_-\\
&+ r^{\fq} \Big( - \frac{V}{2 \fq} + \frac{1}{2 \fq^2}\frac{\nu}{2\pi} - \frac{\fq + 1}{4 \fq^3} \frac{\nu^2}{4\pi^2} V^{-1} + \frac{1}{4 \fq^3} \frac{\nu^3}{8\pi^3} V^{-2}\Big) e^{-\i (1 - \fq) \t1}\omega^3_-, \\
&V^{-1} r^{\fq} e^{\i (\fq + 1)  \t1}\omega^2_- + \frac{1}{\i (\fq + 1)}  V^{-2} r^{\fq} \Big( (\fq + 1) V - \frac{\nu}{2\pi} \Big) e^{\i (\fq + 1 ) \t1}\omega^3_-,\\
&V^{-1} r^{\fq} e^{-\i (\fq + 1)  \t1}\omega^2_- - \frac{1}{\i (\fq + 1)}  V^{-2} r^{\fq} \Big( (\fq + 1) V - \frac{\nu}{2\pi} \Big) e^{-\i (\fq + 1 ) \t1}\omega^3_- \Big\}.	
\end{split}
\end{align*}	

\item For $\fq = - 1$,  
\begin{align*}
 \begin{split}
\cB_{-1}^{2,-} \equiv & \ \Big\{ r^{-1} e^{-\i \theta_1} \omega_-^1, r^{-1} e^{\i \theta_1} \omega_-^1, \quad r^{-1} \omega^2_-,  \quad V^{-2} r^{-1} \omega^3_-\\
&\frac{\i}{2} r^{-1} \Big(V + \frac{3\nu}{4\pi} +  \frac{\nu^2}{4\pi^2} V^{-1} \Big) e^{2\i  \t1}\omega^2_-
+ r^{-1} \Big( \frac{V}{2} +  \frac{\nu}{4\pi} - \frac{\nu^3}{32\pi^3} V^{-2}\Big) e^{2\i\t1}\omega^3_-,  \\  
&-\frac{\i}{2} r^{-1} \Big(V + \frac{3\nu}{4\pi} +  \frac{\nu^2}{4\pi^2} V^{-1} \Big) e^{-2\i  \t1}\omega^2_-
+ r^{-1} \Big( \frac{V}{2} +  \frac{\nu}{4\pi} - \frac{\nu^3}{32\pi^3} V^{-2}\Big) e^{-2\i\t1}\omega^3_- \Big\}.
\end{split}
\end{align*}	 

\item For $\fq = 0$, 
\begin{align*}
 \begin{split}
\cB_0^{2,-} \equiv & \ \Big\{ \omega_-^1, \quad 	Ve^{\i \theta_1}\omega_-^1, \quad 	Ve^{-\i \theta_1}\omega_-^1\\
&\i \Big(\frac{1}{3} \frac{2\pi}{\nu} V^2 + \frac{ 1 }{2} V + \frac{1}{2} \frac{\nu}{2\pi} \Big) e^{\i \t1}\omega^2_- + \Big(\frac{1}{3} \frac{2\pi}{\nu} V^2 + \frac{1}{6} V \Big)  e^{\i \t1} \omega^3_-,	\\
&-\i \Big(\frac{1}{3} \frac{2\pi}{\nu} V^2 + \frac{ 1 }{2} V + \frac{1}{2} \frac{\nu}{2\pi} \Big) e^{-\i \t1}\omega^2_- + \Big(\frac{1}{3} \frac{2\pi}{\nu} V^2 + \frac{1}{6} V \Big)  e^{-\i \t1} \omega^3_-,	\\
&	V^{-1}e^{\i \theta_1} \omega_-^2 + \frac{1}{\i}V^{-2}\left(V- \frac{\nu}{2\pi}\right)e^{\i \theta_1}\omega_-^3, \\
&	V^{-1}e^{-\i \theta_1} \omega_-^2 - \frac{1}{\i}V^{-2}\left(V- \frac{\nu}{2\pi}\right)e^{-\i \theta_1}\omega_-^3 \Big\}.
\end{split}
\end{align*}	 

\item For $\fq = 1$, 
\begin{align*}
 \begin{split}
\cB_1^{2,-} \equiv & \ \Big\{ re^{\i \theta_1}\omega_-^1,	\quad re^{-\i \theta_1}\omega_-^1, \quad r \omega^2_-,\\
& r \Big(\frac{1}{2} V^2 - \frac{\nu}{2\pi} V + \frac{3}{2} \Big(\frac{\nu}{2\pi}\Big)^2 - \frac{3}{2} \Big(\frac{\nu}{2\pi}\Big)^3 V^{-1} + \frac{3}{4} \Big(\frac{\nu}{2\pi}\Big)^4 V^{-2}\Big) \omega^3_-,\\
&V^{-1}r e^{2\i \theta_1}\omega_-^2 + \frac{1}{2\i} V^{-2} r \left(2V - \frac{\nu}{2\pi} \right)	e^{2\i \theta_1}\omega_-^3, \\
&V^{-1}r e^{-2\i \theta_1}\omega_-^2 - \frac{1}{2\i} V^{-2} r \left(2V - \frac{\nu}{2\pi} \right)	e^{-2\i \theta_1}\omega_-^3
\Big\}.
\end{split}
\end{align*}	 
 \end{enumerate}
If a complex-valued 2-form $\omega$ satisfies \eqref{e:S1-invarint-2-form-omega} for some $k\in\dZ$ and solves $\Delta \omega = 0$ on $\widehat{\fM}$ , then 
\begin{align}
\omega \in 
\begin{cases}\bigoplus\limits_{\fq=-2}^2 \mathcal{Z}_{\fq}^{2, -}(\widehat{\fM}), & k \in\{-1,1\}
	\\
	\bigoplus\limits_{\fq=-1}^1 \mathcal{Z}_{\fq}^{2, -}(\widehat{\fM}), & k = 0,
	\\
 \left(\bigoplus\limits_{m=-1}^1 \mathcal{Z}_{k + m}^{2, -}(\widehat{\fM}) \right) \bigoplus \left(\bigoplus\limits_{m=-1}^1 \mathcal{Z}_{-k + m}^{2, -}(\widehat{\fM}) \right),	& k \in \dZ\setminus\{-1,0,1\}.
\end{cases}	
\end{align}
\end{lemma}

\begin{remark} We have that $* \omega = - \omega$ for any ASD 2-form $\omega$, so
\begin{align}
\Delta \omega = d \delta \omega + \delta d \omega
= - d * d * \omega -  * d * d \omega
= d * d \omega -  * d * d \omega
= (\Id - *) d * d \omega.
\end{align}
Therefore, we only need to compute $ d*d\omega$, and then take twice the projection onto the space of ASD $2$-forms.
\end{remark}

\begin{proof}
First, 	we compute
\begin{align}
\begin{split}
d \omega &=   e^{\i k \t1 } \i  k d \t1 \wedge \Big( a(r) \w^1_- + b(r) \w^2_-  + c(r)  \w^3_- \Big)\\
& +  e^{\i k \t1 } \Big( a' dr \wedge \w^1_- + b' dr \wedge \w^2_-  + c' dr \wedge  \w^3_- \Big) + e^{\i k \t1 } \Big( a d \w^1_- + b d \w^2_-  + c d \w^3_- \Big).
\end{split}
\end{align}
Using the above, this is
\begin{align}
\begin{split}
d \omega &=   e^{\i k \t1 } \Big\{ \i k a V^{-1/2} r^{-1} e^1 \wedge \w^1_-
+ \i k b  V^{-1/2} r^{-1}  e^1 \wedge\w^2_-  \\
&+ \i k c  V^{-1/2} r^{-1}  e^1 \wedge \w^3_- +   a' V^{-1/2} e^0 \wedge \w^1_- + b' V^{-1/2} e^0 \wedge \w^2_-  \\
&+ c' V^{-1/2} e^0 \wedge  \w^3_- + b  V^{-1/2} r^{-1} e^0 \wedge e^1 \wedge e^3  -c \Big( \frac{\nu}{ \pi} + V \Big) V^{-3/2} r^{-1} e^0 \wedge e^1 \wedge e^2   \Big\}.
\end{split}
\end{align}
Then we have
\begin{align}
\begin{split}
d \omega &=   e^{\i k \t1 } \Big\{ - \i k a V^{-1/2} r^{-1} e^1 \wedge e^2 \wedge e^3
- \i k b  V^{-1/2} r^{-1}  e^0 \wedge e^1 \wedge e^2 \\
& - \i k c  V^{-1/2} r^{-1}   e^0 \wedge e^1 \wedge e^3  
 -  a' V^{-1/2}   e^0 \wedge e^2 \wedge e^3   + b' V^{-1/2}    e^0 \wedge e^1 \wedge e^3  \\
& - c' V^{-1/2}   e^0 \wedge e^1 \wedge e^2
 + b  V^{-1/2} r^{-1} e^0 \wedge e^1 \wedge e^3  - c \Big( \frac{\nu}{ \pi} + V \Big) V^{-3/2} r^{-1} e^0 \wedge e^1 \wedge e^2   \Big\}.
\end{split}
\end{align}
Collecting terms, we have
\begin{align}
\begin{split}
d \omega &=   e^{\i k \t1 } \Big\{ - \i k a V^{-1/2} r^{-1} e^1 \wedge e^2 \wedge e^3
-  a' V^{-1/2}   e^0 \wedge e^2 \wedge e^3 \\
&+ \Big[  - \i k c  V^{-1/2} r^{-1}   + b' V^{-1/2}  + b  V^{-1/2} r^{-1}  \Big]  e^0 \wedge e^1 \wedge e^3\\
& + \Big[ - \i k b  V^{-1/2} r^{-1} - c' V^{-1/2}
- c  \Big( \frac{\nu}{ \pi} + V \Big) V^{-3/2} r^{-1}  \Big] e^0 \wedge e^1 \wedge e^2   \Big\}.
\label{e:d-ASD}
\end{split}
\end{align}
Now let us define
\begin{align}
p(r) & =   - \i k c  V^{-1/2} r^{-1}   + b' V^{-1/2}  + b  V^{-1/2} r^{-1} \\
q(r) & =  - \i k b  V^{-1/2} r^{-1} - c' V^{-1/2}
- c  \Big( \frac{\nu}{ \pi} + V \Big) V^{-3/2} r^{-1},
\end{align}
so that
\begin{align}
\begin{split}
d \omega =   e^{\i k \t1 } \Big\{ - \i k a V^{-1/2} r^{-1} e^1 \wedge e^2 \wedge e^3
&-  a' V^{-1/2}   e^0 \wedge e^2 \wedge e^3\\
&+ p  e^0 \wedge e^1 \wedge e^3
 + q e^0 \wedge e^1 \wedge e^2   \Big\}.
\end{split}
\label{e:d-omega}
\end{align}

Then we have
\begin{align}
\begin{split}
* d \omega &=   e^{\i k \t1 } \Big\{  \i k a V^{-1/2} r^{-1} e^0
-  a' V^{-1/2}   e^1 - p  e^2  + q  e^3   \Big\}.
\end{split}
\end{align}
So we have that
\begin{align}
\begin{split}
d * d \omega &=   e^{\i k \t1 } \i k d \t1 \wedge \Big\{  \i k a V^{-1/2} r^{-1} e^0
-  a' V^{-1/2}   e^1 - p  e^2 + q e^3   \Big\}\\
& +   e^{\i k \t1 } \Big\{  -  \p_r (a' V^{-1/2}) dr \wedge    e^1 -  a' V^{-1/2}  \p_r(V^{1/2}r) V^{-1} r^{-1} e^0 \wedge e^1 \\
&- p'  dr \wedge e^2 - p  \p_r (V^{1/2}) V^{-1} e^0 \wedge e^2   \\
& + q' dr \wedge e^3  + q
\Big(  \p_r ( V^{-1/2}) e^0 \wedge e^3 + \frac{\nu}{2\pi}V^{-3/2} r^{-1} e^1 \wedge e^2\Big)
\Big\}.
\end{split}
\end{align}
Simplifying some,
\begin{align}
\begin{split}
d * d \omega &=   e^{\i k \t1 } \Big\{  k^2 a V^{-1} r^{-2} e^0 \wedge e^1  - \i k p V^{-1/2} r^{-1} e^1 \wedge  e^2
+ \i k q V^{-1/2} r^{-1} e^1 \wedge   e^3  \\
&  -  \p_r (a' V^{-1/2}) V^{-1/2} e^0 \wedge    e^1 -  a' V^{-1/2}  \p_r(V^{1/2}r) V^{-1} r^{-1} e^0 \wedge e^1 \\
&- p'   V^{-1/2}e^0 \wedge e^2 - p  \p_r (V^{1/2}) V^{-1} e^0 \wedge e^2   \\
& + q'  V^{-1/2} e^0 \wedge e^3  + q \Big(  \p_r ( V^{-1/2}) e^0 \wedge e^3 + \frac{\nu}{2\pi}V^{-3/2} r^{-1} e^1 \wedge e^2\Big)
\Big\}.
\end{split}
\end{align}
Collecting terms,
\begin{align}
\begin{split}
d * d \omega &=   e^{\i k \t1 } \Big\{
\Big[  k^2 a V^{-1} r^{-2}   -  \p_r (a' V^{-1/2}) V^{-1/2}  -  a' V^{-1/2}  \p_r(V^{1/2}r) V^{-1} r^{-1} \Big] e^0 \wedge e^1  \\
& + \Big\{  -  p' V^{-1/2}  -p   \p_r (V^{1/2}) V^{-1} \Big\}  e^0 \wedge e^2 \\
& + \i k q  V^{-1/2} r^{-1} e^1 \wedge   e^3 + \Big\{ q' V^{-1/2} +q  \p_r ( V^{-1/2})  \Big\}e^0 \wedge e^3\\
&+ \Big\{ - \i k p  V^{-1/2} r^{-1} +q  \frac{\nu}{2\pi}V^{-3/2} r^{-1}\Big\} e^1 \wedge e^2 \Big\}.
\end{split}
\end{align}
Recall, we need to take the projection onto the ASD part of this. So we have the following equations.
The $\omega^1_-$ coefficient is
\begin{align}
  k^2 a V^{-1} r^{-2}   -  \p_r (a' V^{-1/2}) V^{-1/2}  -  a' V^{-1/2}  \p_r(V^{1/2}r) V^{-1} r^{-1},\end{align}
which simplifies to
\begin{align}
- V^{-1} \Big(a'' +  r^{-1} a' - k^2 r^{-2} a \Big).
\end{align}
So we have the formula
\begin{align}
{\Delta (  e^{\i k \t1} a(r) \omega^1_-)
=  - e^{\i k \t1} (a'' +  r^{-1} a' - k^2 r^{-2} a) V^{-1}\omega^1_- . }
\label{Expansion-2-form-ASD1}
\end{align}
Recall that solutions of the ODE
\begin{align}
a'' +  r^{-1} a' - k^2 r^{-2} a = 0
\end{align}
are given by
\begin{align}
a(r) =
\begin{cases}
C_1 + C_2 V & k = 0\\
C_1 r^k + C_2 r^{-k} & k \neq 0.
\end{cases}
\end{align}

These prove the $\omega^1_-$ part in Lemma \ref{l:Z2-}.
Next, the $\omega^2_-$ component is
\begin{align}
\begin{split}
(\Id - *) e^{\i k \t1 } &\Big\{   \Big(  -  p' V^{-1/2}  -p   \p_r (V^{1/2}) V^{-1} \Big)  e^0 \wedge e^2 + \i k q  V^{-1/2} r^{-1} e^1 \wedge   e^3 \Big\}\\
& =  e^{\i k \t1 } \Big\{   \Big(  -  p' V^{-1/2}  -p   \p_r (V^{1/2}) V^{-1} \Big) \omega^2_- + \i k q  V^{-1/2} r^{-1} \omega^2_- \Big\}\\
& =  e^{\i k \t1}   \Big(  -  p' V^{-1/2}  - p   \p_r (V^{1/2}) V^{-1}   + \i k q  V^{-1/2} r^{-1}  \Big) \omega^2_- \\
& = - e^{\i k \t1}   \Big(    p'  + \frac{1}{2} (\log(V))' p
  - \i k q  r^{-1}  \Big) V^{-1/2} \omega^2_-.
\end{split}
\end{align}
For the $\omega^3_-$ component, we have
\begin{align}
\notag
(\Id &- *)   e^{\i k \t1 } \Big\{
 \Big( q' V^{-1/2} +q  \p_r ( V^{-1/2})  \Big)e^0 \wedge e^3
+ \Big(- \i k p  V^{-1/2} r^{-1} +q  \frac{\nu}{2\pi}V^{-3/2} r^{-1}\Big) e^1 \wedge e^2 \Big\}\\
\notag
& = e^{\i k \t1 } \Big\{
 \Big( q' V^{-1/2} +q  \p_r ( V^{-1/2})  \Big) \omega^3_-
- \Big(- \i k p  V^{-1/2} r^{-1} +q  \frac{\nu}{2\pi}V^{-3/2} r^{-1}\Big) \omega^3_- \Big\}\\
\notag & =   e^{\i k \t1 } \Big\{ q' V^{-1/2} + q  \p_r ( V^{-1/2})
+ \i k p  V^{-1/2} r^{-1} - q  \frac{\nu}{2\pi}V^{-3/2} r^{-1}\Big\} \omega^3_- \\
\notag & =   e^{\i k \t1 } \Big\{ q'  - \frac{3}{2}  \frac{ \nu}{2\pi}V^{-1} r^{-1} q + \i k p  r^{-1} \Big\} V^{-1/2} \omega^3_- \\
& =   e^{\i k \t1 } \Big\{ q'  - \frac{3}{2} (\log(V))' q + \i k p  r^{-1} \Big\} V^{-1/2} \omega^3_-.
\end{align}
Consequently, we have the formula
\begin{align}
\begin{split}
 \Delta ( e^{\i k \t1} (  b(r) \omega^2_- +  c(r) \omega^3_-))
&= e^{\i k \t1} \Big\{   - \Big(    p'  + \frac{1}{2} (\log(V))' p   - \i k q  r^{-1}  \Big) V^{-1/2}\omega^2_-\\
& \ + \Big( q'  - \frac{3}{2}  (\log(V))' q + \i k p  r^{-1} \Big) V^{-1/2} \omega^3_- \Big\}.
\end{split}
\end{align}
The harmonic condition becomes a first order system for $p$ and $q$:
\begin{align}
\p_r p +  \frac{1}{2}  (\log(V))' p
- \i k q r^{-1} &= 0 \\
 \p_r q - \frac{3}{2} (\log(V))'  q + \i k p  r^{-1}
& = 0.
\end{align}
Define
\begin{align}
\tilde{p} = V^{1/2}  \cdot p,\quad 
\tilde{q}  = V^{-3/2} \cdot  q.
\end{align}
Then the equations become
\begin{align}
\tilde{p}' - \i k r^{-1} V^2  \tilde{q} &= 0 \\
\tilde{q}' +  \i k r^{-1} V^{-2}\tilde{p}  & = 0.
\end{align}
If $k = 0$, then we have
\begin{align}
(\tilde{p}, \tilde{q}) = (C_1, C_2),
\end{align}
so
\begin{align}
(p,q) = ( C_1 V^{-1/2}, C_2 V^{3/2}).
\end{align}
So next, we assume that $k \neq 0$.
To solve this, let us differentiate the second equation to get
\begin{align}
 \tilde{q}'' + \i k (r^{-1} V^{-2})' \tilde{p} + \i k r^{-1} V^{-2} \tilde{p}' = 0.
\end{align}
Using the system again, this is
\begin{align}
\begin{split}
0 &=  \tilde{q}'' + \i k (r^{-1} V^{-2})' \frac{1}{\i k} r V^2 (- \tilde{q})' +
\i k r^{-1} V^{-2} \i k r^{-1} V^2 \tilde{q} \\
  & = \tilde{q}'' - (r^{-1} V^{-2})' r V^2 \tilde{q}' - k^2 r^{-2} \tilde{q}.
\end{split}
\end{align}
Some simplification yields
\begin{align}
 \tilde{q}'' + \Big( \frac{1}{r} + 2 (\log(V))' \Big) \tilde{q}'
-  k^2 r^{-2} \tilde{q} = 0.
\end{align}
For $k \neq 0$, the solutions are
\begin{align}
\tilde{q} = C_1 r^{-k} V^{-1} + C_2 r^k V^{-1}.
\end{align}
Then we can use the second equation
\begin{align}
\tilde{p} = - \frac{1}{\i k} r V^2 \tilde{q}'
\end{align}
to solve for $\tilde{p}$.
This yields
\begin{align}
\begin{split}
\tilde{p} &=    - \frac{1}{\i k} r V^2 \Big(  (-k C_1 r^{-k-1} + k C_2 r^{k-1})V^{-1}
- ( C_1 r^{-k} + C_2 r^k) V^{-2} \frac{ \nu}{2\pi} r^{-1}\Big)\\
& =   - \frac{1}{\i k} \Big(  (-k C_1 r^{-k} + k C_2 r^{k})V
-   \frac{ \nu}{2\pi}( C_1 r^{-k} + C_2 r^{k})  \Big)\\
& =  - \frac{1}{\i k} \Big\{     C_1 \Big( -k V -   \frac{ \nu}{2\pi}\Big) r^{-k}
+ C_2 \Big( k V -   \frac{ \nu}{2\pi} \Big) r^{k} \Big\}.
\end{split}
\end{align}
So to summarize, we have
\begin{align}
(\tilde{p}, \tilde{q}) = (C_1, C_2),
\end{align}
if $k = 0$, and
\begin{align}
(\tilde{p}, \tilde{q}) &=
 C_1\Big(  \frac{\i}{k}  \Big( -k V -   \frac{ \nu}{2\pi}\Big) r^{-k}, r^{-k}V^{-1} \Big)  + C_2 \Big(  \frac{\i}{k}  \Big( k V -   \frac{ \nu}{2\pi} \Big) r^{k}
, r^k V^{-1} \Big)
\end{align}
if $k \neq 0$.

Next, recall the main system is
\begin{align}
 - \i k c  V^{-1/2} r^{-1}   + b' V^{-1/2}  + b  V^{-1/2} r^{-1} & =  p, \\
 - \i k b  V^{-1/2} r^{-1} - c' V^{-1/2} - c  \Big( \frac{\nu}{ \pi} + V \Big) V^{-3/2} r^{-1} & =   q.
\end{align}
The system for $b$ and $c$ then becomes
\begin{align}
\label{hs1}
b' + r^{-1} b - \i k   r^{-1}c  & =  V^{1/2} p,  \\
\label{hs2}
c'  + c  \Big( \frac{\nu}{ \pi} + V \Big) V^{-1} r^{-1}
+  \i k   r^{-1}b & =  - V^{1/2} q.
\end{align}
We let
\begin{align}
\tilde{b} = r b, \quad \tilde{c} = V^2 r c.
\end{align}
Then the system becomes
\begin{align}
\tilde{b}' - \i k V^{-2} r^{-1} \tilde{c} &= V^{1/2} r p, \\
\tilde{c}' + \i k V^2 r^{-1}\tilde{b} &= - V^{5/2} r q.
\end{align}

\subsubsection{$k = 0$} In this case, we know from above that $(p,q) = (C_1 V^{-1/2}, C_2 V^{3/2})$, so the system becomes
 \begin{align}
\tilde{b}'  = C_1 r, \quad 
\tilde{c}' = - C_2 V^{4} r .
\end{align}
The solution is
\begin{align}
\tilde{b} & = A_1  + \frac{C_1}{2} r^2, \\
\tilde{c} & = A_2 - C_2 r^2 \Big(\frac{1}{2} V^4 - \frac{\nu}{2\pi} V^3 + \frac{3}{2} \Big(\frac{\nu}{2\pi}\Big)^2 V^2 - \frac{3}{2} \Big(\frac{\nu}{2\pi}\Big)^3 V + \frac{3}{4} \Big(\frac{\nu}{2\pi}\Big)^4\Big).
\end{align}
So then
\begin{align}
b &= A_1 r^{-1} + \frac{C_1}{2} r, \\
c & = A_2 V^{-2} r^{-1} - C_2 r \Big(\frac{1}{2} V^2 - \frac{\nu}{2\pi} V + \frac{3}{2} \Big(\frac{\nu}{2\pi}\Big)^2 - \frac{3}{2} \Big(\frac{\nu}{2\pi}\Big)^3 V^{-1} + \frac{3}{4} \Big(\frac{\nu}{2\pi}\Big)^4 V^{-2}\Big).
\end{align}

\subsubsection{$k \neq 0$}
If $k \neq 0$, then we differentiate the first equation
of
\begin{align}
\tilde{b}' - \i k V^{-2} r^{-1} \tilde{c} &= V^{1/2} r p, \\
\tilde{c}' + \i k V^2 r^{-1}\tilde{b} &= - V^{5/2} r q,
\end{align}
and use the second equation, and the system for $(p,q)$ to get (after some computation)
\begin{align}
\tilde{b}'' + V^{-1} r^{-1} \Big( \frac{\nu}{ \pi} + V \Big)\tilde{b}'
- k^2 r^{-2} \tilde{b} = V^{-1/2} \Big(   \frac{\nu}{ \pi} + 2 V \Big)p.
\end{align}
Let us make the substitution
\begin{align}
\hat{b} = V \tilde{b}
\end{align}
to get the equation
\begin{align}
\hat{b}'' + r^{-1} \hat{b}' -k^2 r^{-2} \hat{b} = V^{1/2}  \Big(   \frac{\nu}{ \pi} + 2 V \Big)p.
\end{align}
The homogeneous equation is
\begin{align}
\hat{b}'' + r^{-1} \hat{b}' -k^2 r^{-2} \hat{b} = 0,
\end{align}
which has solutions
\begin{align}
\hat{b} = A_1 r^k + A_2 r^{-k},
\end{align}
equivalently,
\begin{align}
b = V^{-1} ( A_1 r^{k-1} + A_2 r^{-k-1}).
\label{e:closed-ASD-b}
\end{align}
We can then solve for $c$ using the equation \eqref{hs1} to get
\begin{align}
c = \frac{1}{\i k} \Big( A_1 V^{-2} r^{k-1} \Big( k V - \frac{\nu}{2\pi} \Big) + A_2 V^{-2} r^{-k-1} \Big( -k V -\frac{\nu}{2\pi} \Big) \Big).
\label{e:closed-ASD-c}
\end{align}
Next, we find a particular solution of
\begin{align}
\hat{b}'' + r^{-1} \hat{b}' -k^2 r^{-2} \hat{b} = \Big(   \frac{\nu}{ \pi} + 2 V \Big) \tilde{p},
\end{align}
where
\begin{align}
\tilde{p} &=
 C_1  \frac{\i}{k}  \Big( -k V -   \frac{ \nu}{2\pi}\Big) r^{-k}  + C_2   \frac{\i}{k}  \Big( k V -   \frac{ \nu}{2\pi} \Big) r^{k}.
\end{align}
We define
\begin{align}
h = r^{-k} \hat{b}.
\end{align}
Then after some computation, the equation becomes
\begin{align}
h'' + (2k+1) r^{-1} h' = \Big(   \frac{\nu}{ \pi} + 2 V \Big)r^{-k} \tilde{p}.
\end{align}
Multiplying by $r^{2k+1}$, we may write this as
\begin{align}
(r^{2k+1} h')' &=  \Big(   \frac{\nu}{ \pi} + 2 V \Big)r^{k+1} \tilde{p}.
\end{align}
This can be integrated to yield
\begin{align}
h' = r^{-2k-1} \int   \Big(   \frac{\nu}{ \pi} + 2 V \Big) r^{k+1}\tilde{p} dr.
\end{align}
We can integrate again to obtain
\begin{align}
h = \int  r^{-2k-1} \int   \Big(   \frac{\nu}{ \pi} + 2 V \Big) r^{k+1}\tilde{p} dr.
\end{align}
Converting back to $b$, we have
\begin{align}
b = V^{-1} r^{k-1} h = V^{-1} r^{k-1}  \int  r^{-2k-1} \int   \Big(   \frac{\nu}{ \pi} + 2 V \Big) r^{k+1}\tilde{p} dr.
\end{align}
When $k \neq \pm 1$, we find that
\begin{align}
\begin{split}
b = \frac{C_1 \i}{k} r^{-k+1} \Big(\frac{k V}{2 k - 2} + \frac{ 2 k - 1}{2 (k - 1)^2}\frac{\nu}{2\pi} + \frac{k^2}{4(k - 1)^3} \frac{\nu^2}{4\pi^2} V^{-1} \Big) \\
+\frac{C_2 \i}{k} r^{k+1} \Big(\frac{k V}{2 k + 2} + \frac{ - 2 k - 1}{2 (k + 1)^2}\frac{\nu}{2\pi} - \frac{k^2}{4(k + 1)^3} \frac{\nu^2}{4\pi^2} V^{-1} \Big).
\end{split}
\end{align}
By \eqref{hs1},
\begin{align}
\begin{split}
c = C_1 r^{-k + 1} \Big(\frac{V}{2 k - 2} + \frac{1}{2 (k - 1)^2}\frac{\nu}{2\pi} - \frac{k - 2}{4(k - 1)^3} \frac{\nu^2}{4\pi^2} V^{-1} - \frac{1}{4(k - 1)^3} \frac{\nu^3}{8\pi^3} V^{-2}\Big) \\
+ C_2 r^{k + 1} \Big(\frac{V}{ - 2 k - 2} + \frac{1}{2 (k + 1)^2}\frac{\nu}{2\pi} - \frac{k + 2}{4(k + 1)^3} \frac{\nu^2}{4\pi^2} V^{-1} + \frac{1}{4(k + 1)^3} \frac{\nu^3}{8\pi^3} V^{-2}\Big).
\end{split}
\end{align}
When $ k = 1$,
\begin{align}
b &= - C_1 \i \Big(\frac{1}{3} \frac{2\pi}{\nu} V^2 + \frac{ 1 }{2} V + \frac{1}{2} \frac{\nu}{2\pi} \Big) + C_2 \i r^{2} \Big(\frac{V}{4} - \frac{ 3}{8}\frac{\nu}{2\pi} - \frac{1}{32} \frac{\nu^2}{4\pi^2} V^{-1} \Big),\\
c &= - C_1 \Big(\frac{1}{3} \frac{2\pi}{\nu} V^2 + \frac{1}{6} V \Big) + C_2 r^{2} \Big( -\frac{V}{ 4} + \frac{1}{8}\frac{\nu}{2\pi} - \frac{3}{32} \frac{\nu^2}{4\pi^2} V^{-1} + \frac{1}{32} \frac{\nu^3}{8\pi^3} V^{-2}\Big).
\end{align}
When $k = - 1$,
\begin{align}
b &= - C_1 \i r^{2} \Big(\frac{V}{ 4} - \frac{ 3}{8}\frac{\nu}{2\pi} - \frac{1}{32} \frac{\nu^2}{4\pi^2} V^{-1} \Big) + C_2 \i \Big(\frac{1}{3} \frac{2\pi}{\nu} V^2 + \frac{ 1 }{2} V + \frac{1}{2} \frac{\nu}{2\pi} \Big),\\
c &= C_1 r^{2} \Big( - \frac{V}{4} + \frac{1}{8}\frac{\nu}{2\pi} - \frac{3}{32} \frac{\nu^2}{4\pi^2} V^{-1} + \frac{1}{32} \frac{\nu^3}{8\pi^3} V^{-2}\Big) - C_2 \Big(\frac{1}{3} \frac{2\pi}{\nu} V^2 + \frac{1}{6} V \Big).
\end{align}

Notice that the Laplacian is a real operator, so we also have conjugate solutions. This completes the proof.
\end{proof}

\bibliographystyle{amsalpha}

\bibliography{CVZ_II_references}

\providecommand{\bysame}{\leavevmode\hbox to3em{\hrulefill}\thinspace}
\providecommand{\MR}{\relax\ifhmode\unskip\space\fi MR }
\providecommand{\MRhref}[2]{%
  \href{http://www.ams.org/mathscinet-getitem?mr=#1}{#2}
}
\providecommand{\href}[2]{#2}
\begin{thebibliography}{HKLR87}

\bibitem[And90]{Anderson}
Michael~T. Anderson, \emph{On the topology of complete manifolds of nonnegative
  {R}icci curvature}, Topology \textbf{29} (1990), no.~1, 41--55.

\bibitem[Bar86]{Bartnik}
Robert Bartnik, \emph{The mass of an asymptotically flat manifold}, Comm. Pure
  Appl. Math. \textbf{39} (1986), no.~5, 661--693.

\bibitem[BKN89]{BKN}
Shigetoshi Bando, Atsushi Kasue, and Hiraku Nakajima, \emph{On a construction
  of coordinates at infinity on manifolds with fast curvature decay and maximal
  volume growth}, Invent. Math. \textbf{97} (1989), no.~2, 313--349.

\bibitem[BM11]{BM}
Olivier Biquard and Vincent Minerbe, \emph{A {K}ummer construction for
  gravitational instantons}, Comm. Math. Phys. \textbf{308} (2011), no.~3,
  773--794.

\bibitem[CC19]{CCII}
Gao Chen and Xiuxiong Chen, \emph{Gravitational instantons with faster than
  quadratic curvature decay ({II})}, J. Reine Angew. Math. \textbf{756} (2019),
  259--284.

\bibitem[CC21a]{CCI}
\bysame, \emph{Gravitational instantons with faster than quadratic curvature
  decay. {I}}, Acta Math. \textbf{227} (2021), no.~2, 263--307.

\bibitem[CC21b]{CCIII}
\bysame, \emph{Gravitational instantons with faster than quadratic curvature
  decay ({III})}, Math. Ann. \textbf{380} (2021), no.~1-2, 687--717.

\bibitem[CK02]{CherkisKapustinALG}
Sergey~A. Cherkis and Anton Kapustin, \emph{Hyper-{K}\"{a}hler metrics from
  periodic monopoles}, Phys. Rev. D (3) \textbf{65} (2002), no.~8, 084015, 10.

\bibitem[CS04]{ChuShin}
Hahng~Yun Chu and Joonkook Shin, \emph{Free actions of finite groups on the
  3-dimensional nilmanifold}, Topology Appl. \textbf{144} (2004), no.~1-3,
  255--270.

\bibitem[CVZ20]{CVZ}
Gao Chen, Jeff Viaclovsky, and Ruobing Zhang, \emph{Collapsing {R}icci-flat
  metrics on elliptic {K}3 surfaces}, Comm. Anal. Geom. \textbf{28} (2020),
  no.~8, 2019--2133.

\bibitem[GH78]{GibbonsHawking}
G.~W. Gibbons and S.~W. Hawking, \emph{Gravitational multi-instantons}, Physics
  Letters B \textbf{78} (1978), no.~4, 430--432.

\bibitem[Hei12]{Hein}
Hans-Joachim Hein, \emph{Gravitational instantons from rational elliptic
  surfaces}, J. Amer. Math. Soc. \textbf{25} (2012), no.~2, 355--393.

\bibitem[HHM04]{HHM}
Tam\'{a}s Hausel, Eugenie Hunsicker, and Rafe Mazzeo, \emph{Hodge cohomology of
  gravitational instantons}, Duke Math. J. \textbf{122} (2004), no.~3,
  485--548.

\bibitem[Hit87]{Hitchin}
N.~J. Hitchin, \emph{The self-duality equations on a {R}iemann surface}, Proc.
  London Math. Soc. (3) \textbf{55} (1987), no.~1, 59--126.

\bibitem[HKLR87]{HKLR}
N.~J. Hitchin, A.~Karlhede, U.~Lindstr\"{o}m, and M.~Ro\v{c}ek,
  \emph{Hyper-{K}\"{a}hler metrics and supersymmetry}, Comm. Math. Phys.
  \textbf{108} (1987), no.~4, 535--589.

\bibitem[HSVZ22]{HSVZ}
Hans-Joachim Hein, Song Sun, Jeff Viaclovsky, and Ruobing Zhang,
  \emph{Nilpotent structures and collapsing {R}icci-flat metrics on the {K}3
  surface}, J. Amer. Math. Soc. \textbf{35} (2022), no.~1, 123--209.

\bibitem[Maz91]{Mazzeo}
Rafe Mazzeo, \emph{Elliptic theory of differential edge operators. {I}}, Comm.
  Partial Differential Equations \textbf{16} (1991), no.~10, 1615--1664.

\bibitem[Mel93]{melrose}
Richard~B. Melrose, \emph{The {A}tiyah-{P}atodi-{S}inger index theorem},
  Research Notes in Mathematics, vol.~4, A K Peters, Ltd., Wellesley, MA, 1993.

\bibitem[Min09]{MinerbeALF}
Vincent Minerbe, \emph{A mass for {ALF} manifolds}, Comm. Math. Phys.
  \textbf{289} (2009), no.~3, 925--955.

\bibitem[Sal89]{Salamon}
Simon Salamon, \emph{Riemannian geometry and holonomy groups}, Pitman Research
  Notes in Mathematics Series, vol. 201, Longman Scientific \& Technical,
  Harlow; copublished in the United States with John Wiley \& Sons, Inc., New
  York, 1989.

\end{thebibliography}

 \end{document}